\documentclass[12pt]{amsart}
\usepackage{a4wide}
\usepackage{amsmath,amsfonts,amssymb}
\usepackage{amsthm}
\usepackage{ifthen}
\usepackage{longtable}
\usepackage{array}
\usepackage{url}
\usepackage[utf8]{inputenc}

\usepackage{color}

\newcommand{\Z}{\mathbb{Z}}
\newcommand{\Q}{\mathbb{Q}}
\newcommand{\R}{\mathbb{R}}

\renewcommand{\L}{\mathcal{L}}

\newtheorem*{theorem*}{Theorem}
\newtheorem{theorem}{Theorem}
\newtheorem{lemma}{Lemma}
\newtheorem{corollary}{Corollary}
\newtheorem{proposition}{Proposition}
\newtheorem{conjecture}{Conjecture}

\newtheorem{claim}{Claim}

\theoremstyle{remark}
\newtheorem{remark}{Remark}

\begin{document}
\title[Linear Equations in Recurrence Sequences]{Effective results for linear Equations in Members of two Recurrence Sequences}

\author[V. Ziegler]{Volker Ziegler}
\address{V. Ziegler,
Institute of Mathematics,
University of Salzburg,
Hellbrunnerstrasse 34/I,
A-5020 Salzburg, Austria}
\email{volker.ziegler\char'100sbg.ac.at}
\date{}

\begin{abstract}
Let $(U_n)_{n=0}^\infty$ and $(V_m)_{m=0}^\infty$ be two linear recurrence sequences. For fixed positive integers $k$ and $\ell$, fixed $k$-tuple $(a_1,\dots,a_k)\in \Z^k$
and fixed $\ell$-tuple $(b_1,\dots,b_\ell)\in \Z^\ell$ we consider the linear equation
$$a_1U_{n_1}+\cdots +a_k U_{n_k}=b_1V_{m_1}+\cdots + b_\ell V_{m_\ell}$$
in the unknown non-negative integers $n_1,\dots,n_k$ and $m_1,\dots,m_\ell$. Under the assumption that the linear recurrences $(U_n)_{n=0}^\infty$ and $(V_m)_{m=0}^\infty$
have dominant roots and under the assumption of further mild restrictions we show that this equation has only finitely many solutions which can be found effectively.
\end{abstract}

\maketitle

\section{Introduction}\label{Sec:Intro}

Let $(U_n)_{n=0}^\infty$ be a sequence. We say that $(U_n)_{n=0}^\infty$ is a linear recurrence sequence of order $d>0$ if there exist complex numbers
$c_1, \dots , c_d$ such that
$$U_{n+d}=c_1U_{n+d-1}+\dots + c_d U_{n}.$$
If $c_1 , \dots , c_{d}$ and $U_0, \dots , U_{d-1}$ are all integers, then $U_n$ is an integer for all
$n\geq 0$ and we say that $(U_n)_{n=0}^\infty$ is defined over the integers. In what follows we
will always assume that  $(U_n)_{n=0}^\infty$ is defined over the integers.

It is a well known fact that if the companion polynomial to $(U_n)_{n=0}^\infty$ is of the form 
$$F (X) = X^d - c_1X^{d-1} - \dots - c_{d} = \prod_{i=1}^t (X-\alpha_i)^{m_i},$$
where $\alpha_1 , \dots , \alpha_t$ are distinct complex numbers, and $m_1 , \dots , m_t$ are positive integers
whose sum is $d$, then there exist polynomials $u_1(X), \dots , u_t(X)$ whose coefficients
are in $\Q(\alpha_1 ,\dots , \alpha_t)$ such that $u_i (X)$ is of degree at most $m_i - 1$ for $i = 1, \dots , t$ 
and
\begin{equation}\label{eq:Binet}
 U_n =\sum_{i=1}^t u_i (n)\alpha_i^n
\end{equation} 
holds for all $n\geq 0$. We call a recurrence sequence $(U_n)_{n=0}^\infty$ simple if $m_i=1$ for all $i=1,\dots,t=d$,
non-degenerate if $\alpha_i/\alpha_j$ is not a root of unity for any $1\leq i <j \leq t$ and say that it satisfies the
dominant root condition if $|\alpha_1|>|\alpha_2|\geq |\alpha_3|\geq \dots \geq |\alpha_t|$. Throughout this paper we will assume
that the recurrence sequences are simple, non-degenerate, satisfy the dominant root condition and are defined over the integers. 

General results on linear equations involving recurrence sequences have been made most prominently by Schlickewei and Schmidt \cite{Schlickewei:1993a,Schlickewei:1993}.
Recently the case of linear equations involving Fibonacci numbers and powers of two have been picked up by several authors such as
Bravo and Luca \cite{Bravo:2016} and Bravo, G\'{o}mez and Luca \cite{Bravo:2017} who studied the Diophantine equations $F_n+F_m = 2^a$
and $F_n^{(k)} + F_m^{(k)} = 2^a$ respectively, where $(F_n)_{n=0}^\infty$ and $(F_n^{(k)})_{n=0}^\infty$ denote the sequence of Fibonacci and $k$-Fibonacci numbers.
Besides, Bravo, Faye and Luca \cite{Bravo:2017a} studied the Diophantine equation $P_l + P_m + P_n = 2^a$, where $(P_n)_{n=0}^\infty$ is the sequence of Pell numbers.
Recently Chim and Ziegler \cite{Chim:2017b} generalized the result of Bravo and Luca \cite{Bravo:2016} and solved
completely the Diophantine equations 
$$F_{n_1} + F_{n_2} = 2^{a_1} + 2^{a_2} + 2^{a_3}$$
and
$$F_{m_1} + F_{m_2} + F_{m_3} =2^{t_1} + 2^{t_2}.$$
We also want to mention the results due to Bugeaud, Cipu and Mignotte \cite{Bugeaud:2013} who considered the problem on representing Fibonacci and Lucas numbers in an integer base. Beside other 
things they solved the Diophantine equation
$$F_{n}= 2^{a_1} + 2^{a_2} + 2^{a_3}+2^{a_4}$$
completely (see \cite[Theorem 2.2]{Bugeaud:2013}). 

In this paper we want to further generalize these results and study the Diophantine equation
\begin{equation}\label{eq:main}
 a_1U_{n_1}+\cdots +a_k U_{n_k}=b_1V_{m_1}+\cdots + b_\ell V_{m_\ell},
\end{equation}
where $(U_n)_{n=0}^\infty$ and $(V_m)_{m=0}^\infty$ are fixed linear recurrence sequences defined over the integers and $a_1,\dots,a_k$ and $b_1,\dots,b_\ell$ are fixed non-zero integers.
Under some mild technical restrictions we show that there exist only finitely many effectively computable solutions to \eqref{eq:main}.

Let us mention that the problem of finding all members in a given recurrence sequence that have only few non-zero digits in a given integer base $g$ is closely related to studying Diophantine equations
of the form 
\begin{equation}\label{eq:U_gadic} U_n=b_1 g^{m_1}+\dots+b_\ell g^{m_\ell}. \end{equation}
Note that this is a special case of the Diophantine equation \eqref{eq:main}. However, effective finiteness results for Diophantine equation \eqref{eq:U_gadic}
have been proved by Stewart \cite{Stewart:1980}. In particular, Stewart showed that
the number of non-zero digits of $U_n$ in a given base $g$ is bounded from above by $\frac{\log n}{\log\log n +C}-1$, where $C$ is an effectively computable constant. However our method 
allows us the reprove this classical result in a more general setting (cf. Theorem \ref{th:Numeration}). Let us also remark that our method of proof follows more closely the line of a proof due to Luca
\cite{Luca:2000} who also reproved Stewart's result.

More precisely we aim to generalize the following results due to Senge and Strauss \cite{Senge:1973} who showed that the number of integers, such that
the sum of non-zero digits in each of the bases $g$ and $h$ lies below a fixed bound $M$, is finite if and only if $\frac{\log g}{\log h}$ is irrational.
However, their result which is ineffective has been brought to an effective form by Stewart \cite{Stewart:1980}. As a consequence of our proof we obtain
a generalization of Stewart's result which applies to generalized number systems (cf. Theorem \ref{th:Numeration}).

Inspired by the results due to Bugeaud, Cipu and Mignotte \cite{Bugeaud:2013} and Bravo and Luca \cite{Bravo:2016} we consider the problem of finding all
integers that have few non-zero integer digits in their binary as well as in their Zeckendorf expansion
(cf. Theorems \ref{th:Zeckendorf-Binary} and \ref{th:reduction}). 

In the next section we will introduce some further notations and will state our main results. The proof of our main Theorems \ref{th:finiteness} and \ref{th:Numeration} will be
established by an induction argument. In order to make the Theorems \ref{th:finiteness} and \ref{th:Numeration} accessible to induction we state a technical theorem, Theorem \ref{th:fundamental},
which will imply Theorems \ref{th:finiteness} and \ref{th:Numeration}. In Section \ref{Sec:Growth} we discuss the technical restrictions under which our
results hold. In the short Section \ref{Sec:Deduction} we demonstrate how Theorems~\ref{th:finiteness} and \ref{th:Numeration} follow from the technical Theorem \ref{th:fundamental}.
After stating some auxiliary results such as lower bounds for linear forms in logarithms and results on heights of algebraic numbers (see Section \ref{Sec:Aux}),
we prove Theorem~\ref{th:fundamental} in Section~\ref{Sec:Proof}. In the final two sections we discuss how to use our results in practice
to solve equations of the form \eqref{eq:main} by discussing classical reduction methods as Baker-Davenport reduction and the usage of continued fractions. 
We demonstrate the strength of our method by proving a result on the number of non-zero binary digits and non-zero digits
in the Zeckendorf expansion (see Theorem \ref{th:Zeckendorf-Binary} below) in the last section. 

\section{Notations and statement of the Main results}\label{Sec:Notations}

For a given recurrence sequence $(U_n)_{n=0}^\infty$ we say that a $k$-tuple $(a_1,\dots,a_k)$ of non-zero integers admits \emph{dominance} if for 
every $k$-tuple of non-negative integers $n_1>n_2>\dots>n_k\geq 0$ we have that
$$\left|a_1U_{n_1}+\cdots+a_kU_{n_k}\right|\gg \left|U_{n_1}\right|,$$
where the implied constant does not depend on $n_1,\dots,n_k$. How to recognize whether a given $k$-tuple $(a_1,\dots,a_k)$ admits dominance will be discussed in the next
section. 

\begin{theorem}\label{th:finiteness}
 Let $(U_n)_{n=0}^\infty$ and $(V_m)_{m=0}^\infty$ be non-constant, simple, non-degen\-erate, linear recurrence sequences defined over the integers with dominant roots $\alpha$ and $\beta$ respectively 
 such that $\alpha$ and $\beta$ are multiplicatively independent. Assume that the $k$-tuple $(a_1,\dots,a_k)$ and the $\ell$-tuple $(b_1,\dots,b_\ell)$ of non-zero integers
 admit dominance for $(U_n)_{n=0}^\infty$ and $(V_m)_{m=0}^\infty$ respectively. Then there exists an effectively computable constant $N$ such that every solution
 $(n_1,\dots,n_k,m_1,\dots,m_\ell)$ to equation~\eqref{eq:main} with $n_1>\dots>n_k\geq 0$ and $m_1>\dots>m_\ell\geq 0$ satisfies $\max\{n_1,m_1\}\leq N$. 
\end{theorem}

\begin{remark}
 The assumption that the $k$-tuple $(a_1,\dots,a_k)$ and the $\ell$-tuple $(b_1,\dots,b_\ell)$ admit dominance is essential in view of the
 finiteness of solutions. Indeed consider the example, where $(U_n)_{n= 0}^\infty=(F_n)_{n=0}^\infty$ is the Fibonacci sequence (i.e. the sequence that satisfies $F_0=1, F_1=2$ and $F_n=F_{n-1}+F_{n-2}$ for all $n\geq 2$)
 and $(V_m)_{m=0}^\infty$ is the sequence given by $V_m=2^m$ for all $m\geq 0$. Let $k=4$, $\ell=1$, $(a_1,a_2,a_3,a_4)=(1,-1,-1,1)$ and $b_1=1$, i.e.
 we consider the Diophantine equation
 $$F_{n_1}-F_{n_2}-F_{n_3}+F_{n_4}=2^m.$$
 Then $(n_1,n_2,n_3,n_4,m)=(t,t-1,t-2,4,3)$ is a solution to the above equation for every $t\geq 2$. Note that $(1,-1,-1,1)$ does not admit dominance for the Fibonacci sequence. 
\end{remark}

Theorem \ref{th:finiteness} has some remarkable consequences for generalized numeration systems $G$ with base sequence $(G_n)_{n=0}^\infty$. In particular, let $(G_n)_{n=0}^\infty$ be
a non-constant, simple, non-degenerate, linear recurrence of order $d$ with
$$G_{n+d} = c_1G_{n+d-1} + \cdots + c_{d}G_n ,$$
where $d \geq 1$, $G_0 = 1$, $G_k = c_1 G_{k-1} + \cdots + c_{k} G_0 + 1$ for $k < d$ and for $i=1,\dots,d$ the $c_i$'s are non-negative integers. Then every positive integer $n$ can be represented as
\begin{equation}\label{eq:repr}
n=\sum_{k=0}^{\infty} \varepsilon_k(n) G_k,
\end{equation}
where $0 \leq \varepsilon_k(n) < G_{k+1}/G_k $ and  $\varepsilon_k(n) \in \Z$.
This expansion (called $G$-expansion) is uniquely determined provided that
\begin{equation}\label{eq:regular}
\sum_{k=0}^{K-1}\varepsilon_k(n) G_k < G_K\qquad \text{for all}\;\; K>0,
\end{equation}
where the digits $\varepsilon_k (n)$ are computed by the greedy algorithm (see for instance \cite{Fraenkel:1985}).
We call a sequence of digits $(\varepsilon_0,\varepsilon_{1}, \ldots)$ regular if it satisfies \eqref{eq:regular} for every $K$. From now on we assume that all $G$-expansions are regular.

Furthermore we denote by
$$H_G(n)=\sum_{k=1}^{\infty}(\varepsilon_k(n))^0$$
the Hamming weight of $n$ in base $G$, i.e. the number of non-zero $G$-adic digits of $n$, where $n$ has regular $G$-expansion \eqref{eq:repr}. Note that since the uniqueness of regular
$G$-expansion the Hamming weight $H_G(n)$ does not depend on the $G$-expansion and is therefore well defined. Also note that any $k$-tuple $(\varepsilon_{n_1}(n),\dots,\varepsilon_{n_k}(n))$ 
of non-zero digits of a regular $G$-expansion admits dominance for $(G_n)_{n=0}^\infty$. Finally let us remark that the companion polynomial 
$$X^d-c_1X^{d-1}-\cdots-c_{d} $$
of $(G_n)_{n=0}^\infty$ has a real, dominant (characteristic)
root $\alpha$ with $c_1 < \alpha < c_2$, by a simple application of Rouch\'e's theorem, the conjugated
roots have to be of smaller modulus. Furthermore let us assume that the base sequence $(G_n)_{n=0}^\infty$ is simple and non-degenerate. In particular, this is fulfilled whenever the
coefficients satisfy $c_1\geq c_2\geq \dots \geq c_{d}\geq 1$. For more details on generalized numeration systems see e.g. \cite{Grabner:1995}.

With these notations and remarks on hand we immediately deduce from Theorem~\ref{th:finiteness} the following corollary:

\begin{corollary}\label{cor:Numeration}
 Given expansions $G$ and $H$ such that the dominant roots of the corresponding recurrence relations are multiplicatively independent. Then for every $M\geq 2$ there exists an
 effectively computable constant $\tilde N$ such that
 $$H_G(n)+H_H(n)\leq M$$
 implies that $n<\tilde N$.
\end{corollary}

But even more is true. We can find a uniform formula for the upper bound $\tilde N$:

\begin{theorem}\label{th:Numeration}
 Given expansions $G$ and $H$ such that the dominant roots of the corresponding recurrence relations are multiplicatively independent. Then there exists an effectively computable constant $\tilde C$ (depending on
 $G$ and $H$) such that
 $$H_G(n)+H_H(n)\leq M$$
 implies that 
 $$\log n\leq\left(\tilde C M\log M\right)^{M-1}.$$
\end{theorem}

Note that in the case that $G$ and $H$ are classical bases of number systems, say $G_n=g^n$ and $H_n=h^n$, we obtain exactly Stewart's famous result \cite{Stewart:1980}.

In order to demonstrate our method we consider the case where the base $G$ is the Fibonacci sequence, i.e. we consider the Zeckendorf expansion on the one hand, and $H$ is the sequence $H_n=2^n$, 
i.e. we consider the binary expansion on the other hand. In view of this application the attentive reader will recognize that this is the reason why we chose the not quite common definition of
the Fibonacci sequence, i.e. the choice $F_0=1$ and $F_1=2$. With this choice the Fibonacci sequence satisfies the requirements for $G$-bases as described above. In this particular case we obtain:

\begin{theorem}\label{th:Zeckendorf-Binary}
 Let $H_Z(n)$ be the Hamming weight of the Zeckendorf expansion of $n$ and $H_b(n)$ the Hamming weight of the binary digit expansion of $n$. If
 $$H_Z(n)+H_b(n)\leq M,$$
 then we have that 
 $$\log n<\left(8.23 \cdot 10^{15} M\log M\right)^{M-1}.$$
 Additionally we have:
 \begin{itemize}
  \item If $M=2$, then $n=1,2,8$.
  \item If $M=3$, then $n=3,4,5,16,34,144$.
  \item If $M=4$, then 
  $$n=6,9,10,13,18,21,24,32,36,64,68,256,288,1024.$$
  \item If $M=5$, then 
  \begin{align*}
   n=& \: 11,12,14,17,20,22,26,35,37,40,42,48,65,66,76,89,96,97,128,136,\\
       &\: 145,146,152,160,257,272, 322,384,385,521,576,610,644,1026,1042,\\
     &\:  1152,1600,2584,2592.
  \end{align*}
 \end{itemize}
\end{theorem}

In order to derive Theorems \ref{th:finiteness} and \ref{th:Numeration} we prove that an even stronger statement holds.
The formulation of this stronger statement allows us to use an induction argument to proof Theorem~\ref{th:fundamental} stated below and therefore also
prove Theorems \ref{th:finiteness} and \ref{th:Numeration}.

In order to formulate the next result we have to introduce some further notations. We put $n_K=m_L=0$ for $K>k$ and $L>\ell$. 
Moreover we denote by $\mathcal I^{(U)}_{a_1,\dots,a_k}$ the infimum 
$$\mathcal I^{(U)}_{a_1,\dots,a_k}= \inf_{n_1>\dots>n_k} \left\{\frac{\left|a_1U_{n_1}+\cdots+a_kU_{n_k}\right|}{|U_{n_1}|}\right\}$$
and by $\mathcal S^{(U)}_{a_1,\dots,a_k}$ the supremum
$$\mathcal S^{(U)}_{a_1,\dots,a_k}= \sup_{n_1>\dots>n_k} \left\{\frac{\left|a_1U_{n_1}+\cdots+a_kU_{n_k}\right|}{|U_{n_1}|}\right\}$$
Then we have the following theorem:

\begin{theorem}\label{th:fundamental}
Let the assumptions of Theorem \ref{th:finiteness} be in force and assume that $n_1\geq 3$. Then there exists a constant $C$ such that for every $m\geq 4$ there exists a pair $(K,L)$ 
of positive integers $K,L\geq 2$ with $K+L=m$ such that this pair $(K,L)$ satisfies for every
solution $(n_1,\ldots,n_k,m_1,\ldots,m_\ell)$ to Diophantine equation~\eqref{eq:main} 
\begin{equation}\label{eq:min}
 \min\{n_1-n_K,m_1-m_L\}\leq (C\log n_1)^{m-3}
\end{equation}
and
\begin{equation}\label{eq:max}
 \max\{n_1-n_{K-1},m_1-m_{L-1}\}\leq (C \log n_1)^{m-4}.
\end{equation}
Moreover, to the pair $(K,L)$ there exists a sequence of pairs $(K_i,L_i)$ with $i=4,\dots, m$ such that $K_i+L_i=i$, $2\leq K_4\leq \dots \leq K_m=K$ and $2\leq L_4\leq \dots \leq L_m=L$
such that
\begin{equation}\label{eq:seq}
 \max\{n_1-n_{K_i-1},m_1-m_{L_i-1}\}\leq (C \log n_1)^{i-4} \quad 4\leq i\leq m-1.
\end{equation}
The constant $C$ depends on $(U_n)_{n=0}^\infty$, $(V_m)_{m=0}^\infty$, $\mathcal I^{(U)}_{a_1,\dots,a_k}$, $\mathcal I^{(V)}_{b_1,\dots,b_\ell}$, $\mathcal S^{(U)}_{a_1,\dots,a_k}$, $\mathcal S^{(V)}_{b_1,\dots,b_\ell}$,
$A=\max_{1\leq i \leq k}\{|a_i|\}$ and $B=\max_{1\leq i \leq \ell}\{|b_i|\}$ but not on $M=k+\ell$.
\end{theorem}

In his book \cite[Conjecture 14.25]{Waldschmidt:DA} Waldschmidt conjectured very sharp lower bounds for linear forms in logarithms. If we would apply these conjectural lower bounds
due to Waldschmidt instead of the lower bounds due to Baker and Wüstholz \cite{bawu93}, which are asymptotically the best known bounds, we would obtain instead of inequalities \eqref{eq:min} and \eqref{eq:max}
the inequalities 
$$
 \min\{n_1-n_K,m_1-m_L\}\leq C^{m-3} \log n_1 
$$
and
$$
 \max\{n_1-n_{K-1},m_1-m_{L-1}\}\leq C^{m-4}\log n_1 .
$$
These conjectural bounds would lead to the following stronger version of Theorem \ref{th:Numeration}.

\begin{conjecture}
 Given expansions $G$ and $H$ such that the dominant roots of the corresponding recurrence relations are multiplicatively independent.
 Then there exists an effectively computable constant $\tilde C$ (depending on $G$ and $H$) such that
 $$H_G(n)+H_H(n)\leq M$$
 implies that 
 $$\log n\leq \tilde C^M.$$
\end{conjecture}

For the rest of the paper we denote by $C_1,C_2,\dots$ constants, which are effectively computable and depend only
on $\mathcal I^{(U)}_{a_1,\dots,a_k}$, $\mathcal I^{(V)}_{b_1,\dots,b_\ell}$, $\mathcal S^{(U)}_{a_1,\dots,a_k}$, $\mathcal S^{(V)}_{b_1,\dots,b_\ell}$,
$A=\max_{1\leq i \leq k}\{|a_i|\}$, $B=\max_{1\leq i \leq \ell}\{|b_i|\}$, $(U_n)_{n=0}^\infty$ and $(V_m)_{m=0}^\infty$ but not on $M=k+\ell$. In several cases we also consider constants of the form $C_2^{(U)}$ and $C_2^{(V)}$ which
indicate that the constant depends only on the sequence $(U_n)_{n=0}^\infty$ and $(V_m)_{m=0}^\infty$ respectively. Since the proof of the theorem is already
rather technical, we abandon to keep track of the explicit values of these constants to avoid further technical details. The interested and experienced reader will not have problems
to compute these constants for concrete examples explicitly. We demonstrate these computations during the proof of Theorem~\ref{th:Zeckendorf-Binary} in Section \ref{Sec:Examples}.

\section{Notes on the growth condition}\label{Sec:Growth}

The aim of this section is twofold. Firstly we want to describe an effective procedure to decide whether a $k$-tuple $(a_1,\dots,a_k)$ admits dominance
for a simple, non-degenerate sequence $(U_n)_{n=0}^\infty$ defined over the integers. Secondly we want to give a characterization of dominance which will yield
useful consequences in view of the proof of our main result Theorem \ref{th:fundamental}.

Before we can settle our first goal, let us prove a useful lemma. Therefore let
$$X^{d_U}-c_1X^{d_U-1}-\dots-c_{d_U}$$
be the characteristic polynomial of $(U_n)_{n=0}^\infty$ with roots $\alpha_1=\alpha,\alpha_2,\dots,\alpha_{d_U}$ such that $|\alpha|>|\alpha_2|\geq \dots \geq |\alpha_{d_U}|$ and let
$$X^{d_V}-d_1X^{d_V-1}-\dots-d_{d_V}$$
be the characteristic polynomial of $(V_m)_{m=0}^\infty$ with roots $\beta_1=\beta,\beta_2,\dots,\beta_{d_V}$ such that $|\beta|>|\beta_2|\geq \dots \geq |\beta_{d_V}|$.
Then by our assumption that $(U_n)_{n=0}^\infty$ and $(V_m)_{m=0}^\infty$ are simple there exist algebraic numbers $u_1=u,u_2,\dots,u_{d_U}$ each of degree at most $d_U$ and contained in $\Q(\alpha_1,\dots,\alpha_{d_U})$
and algebraic numbers $v_1=v,v_2,\dots,v_{d_V}$ each of degree at most $d_V$ and contained in $\Q(\beta_1,\dots,\beta_{d_V})$ such that
\begin{equation}\label{eq:Binet-UV}
 \begin{split}
  U_n&= u\alpha^n+\sum_{j=2}^{d_U} u_j\alpha_j^n\\
  V_n&= v\beta^n+\sum_{j=2}^{d_V} v_j\beta_j^n.
 \end{split}
\end{equation}
Let us also note that by our assumption that $(U_n)_{n=0}^\infty$ is non-degenerate and defined over the integers, the dominant root $\alpha$ is a real algebraic integer which is not a root of unity,
hence we have $|\alpha| > 1$. 

\begin{lemma}\label{lem:approx}
Under the assumptions of Theorem \ref{th:finiteness} there exist constants $C_1^{(U)}$ and $C_1^{(V)}$ such that
 $$|U_n-u\alpha^n|<C_1^{(U)} |\alpha_2|^n$$
and
$$|V_m-v\beta^m|<C_1^{(V)} |\beta_2|^m.$$
\end{lemma}

Let us address to the question: How can we effectively decide whether $(a_1,\dots,a_k)$ admits dominance? This question is answered by the following proposition:

\begin{proposition}\label{prop:stable-growth}
 The $k$-tuple $(a_1,\dots,a_k)$ does not admit dominance for $(U_n)_{n=0}^\infty$ if and only if there exists some $1\leq K\leq k$ and some integers 
 $n_1>\dots>n_K\geq 0$ such that
 \begin{equation}\label{eq:not-stable}
  a_1 \alpha^{n_1}+\dots+a_K \alpha^{n_K}=0.
 \end{equation}
 It can be effectively decided whether such $K$ and such integers $n_1>\dots>n_K\geq 0$ exist, and in case of their existence such an instance can be effectively computed.
 
 Moreover, in case that $(a_1,\dots,a_k)$ admits dominance there exist positive, effectively computable constants $C_2=C_2^{(U)}$ and $C_3=C_3^{(U)}$ such that
 \begin{equation}\label{eq:lower_bound_1}
  |a_1 \alpha^{n_1}+\dots+a_K \alpha^{n_K}|>C_2 |\alpha|^{n_1}
 \end{equation}
 for any $1\leq K \leq k$ and any integers $n_1>\dots>n_K\geq 0$ and
 \begin{equation}\label{eq:lower_bound_2}
  |a_1 U_{n_1}+\dots+a_k U_{n_k}|>C_3 |U_{n_1}|
 \end{equation}
 for any integers $n_1>\dots>n_k\geq 0$. 
\end{proposition}

\begin{remark}\label{rem:lower-bound-numeration}
 Note that in general the constants $C_2$ and $C_3$ depend on $(U_n)_{n=0}^\infty$, $k$ and $A=\max_{1\leq i \leq k}\{|a_i|\}$. In view of Theorems \ref{th:Numeration} and \ref{th:fundamental}
 the dependence on $k$ is problematic. However in case that $a_1, \dots, a_k$ come from a regular digit expansion all the $a_i$'s are positive, $\alpha>1$ and $U_n\geq 0$ for all $n$. Therefore one immediately sees that
 in this case Proposition \ref{prop:stable-growth} holds with $C_2=C_3=a_1$.
\end{remark}

\begin{proof}
 For the moment let us assume that there exists an integer $1\leq K\leq k$ and integers  $n_1>\dots>n_K\geq 0$ such that
 $$a_1 \alpha^{n_1}+\dots+a_K \alpha^{n_K}=0.$$
 Then we have due to Lemma \ref{lem:approx} that
  \begin{align*}
  |a_1 U_{n_1}+\cdots +a_kU_{n_k}|\leq & \sum_{i=1}^K |a_i U_{n_i}-u a_i \alpha^{n_i}|+\sum_{i=K+1}^k |a_i U_{n_i}|\\
  \leq &K A C_1^{(U)} |\alpha_2|^{n_1}+(k-K)A\max_{K+1\leq i \leq k}\{|U_{n_i}|\},
 \end{align*}
 where we put $n_{k+1}=0$ in case that $K=k$. It is easy to see that in the case that $n_{K+1}$ is fixed we have 
 $$ \lim_{n_1 \rightarrow \infty}\frac{KA C_1^{(U)} |\alpha_2|^{n_1}+(k-K)A\max_{K+1\leq i \leq k}\{|U_{n_i}|\}}{U_{n_1}}=0,$$
 hence $(a_1,\dots,a_k)$ does not admit dominance for $(U_n)_{n\geq 0}$. We have therefore shown that \eqref{eq:not-stable} is sufficient for not admitting dominance. Therefore we are left to show
 that \eqref{eq:not-stable} is necessary for not admitting dominance.
 
 To prove necessity we proceed by an induction argument on $K$. More precisely we claim the following:
 
 \begin{claim}
 There are recursively computable sets $\mathcal N_2,\dots , \mathcal N_k$ and constants $C_2^{(2)},\dots, C_2^{(k)}$
 such that for every tuple $(m_2,m_3,\dots ,m_K)\not \in \mathcal N_K$ of integers with $0<m_2<\dots<m_K$ we have that for all integers $m_{K+1},\dots,m_k$ with $m_{K}<m_{K+1}<\dots<m_k$ 
 \begin{itemize}
  \item either
 $$|a_1+a_2 \alpha^{-m_2}+a_3 \alpha^{-m_3}+\dots+a_k \alpha^{-m_k}|\geq C_2^{(K)}$$
 \item or there exist integers $n_1>\dots>n_K\geq 0$ such that \eqref{eq:not-stable} is satisfied.
 \end{itemize}
 \end{claim}
 
 We start with the case that $K=2$. If there exist integers $n_1>n_2$ such that $a_1\alpha^{n_1}+a_2\alpha^{n_2}=0$, then we can find $n_1$ and $n_2$ by solving
 the equation $a_1+a_2 \alpha^{-n}=0$ for $n>0$. Therefore we may assume that for any integers $n_1>n_2$ we have that $a_1 \alpha^{n_1}+a_2\alpha^{n_2}\neq 0$.
 Let us denote by $\mathcal N_2$ the set of integers $m$ such that
 $$|a_1|-\frac{A|\alpha|^{-m+1}}{|\alpha|-1} <\dfrac{1}{2}.$$
 Note that $\mathcal N_{2}$ is a finite set and that $m\not\in \mathcal N_{2}$ implies that
$$|a_1+a_2 \alpha^{-m}+a_3 \alpha^{-m_3}+\dots+a_k \alpha^{-m_k}|>|a_1|-\frac{A|\alpha|^{-m+1}}{|\alpha|-1}\geq \dfrac{1}{2}$$
for any choice of integers $m_3,\dots,m_k$ with $m < m_3< \dots <m_k$. If $\mathcal N_2$ is empty we put $\mathcal N_i=\emptyset$ for $i=2,\dots,k$ and $C_2^{(2)}=\dots=C_2^{(k)}=\frac{1}{2}$.
In this case also inequality \eqref{eq:lower_bound_1} holds with $C_2=C_2^{(k)}$. If $\mathcal N_2$ is not empty we compute
$$C_2^{(2)}=\min\left\{\frac{1}{2}, \min_{m\in\mathcal N_2}\{|a_1+a_2\alpha^{-m}|\}\right\}.$$
Thus for any choice of $n_1>n_2\geq 0$ we have  
$$|a_1 \alpha^{n_1}+a_2 \alpha^{n_2}|\geq C_2^{(2)} |\alpha|^{n_1}.$$
Therefore the case $K=2$ is settled.

Assume now that we have computed the finite sets $\mathcal N_i$ and constants $C_2^{(i)}$ for all $i<K$ such that
$(m_2,m_3,\dots,m_i)\not\in\mathcal N_i$ implies that for all $m_{i+1}<\dots<m_k$ with $m_i<m_{i+1}$ we have
 $$|a_1+a_2 \alpha^{-m_2}+a_3 \alpha^{-m_3}+\dots+a_k \alpha^{-m_k}|\geq C_2^{(i)}.$$
Furthermore we may assume that for any
choice of $n_1>\dots>n_i\geq 0$ we have that
$$a_1 \alpha^{n_1}+\dots+a_i \alpha^{n_i}\neq 0.$$
Then we can find for all $(m_2,\dots,m_{K-1})\in \mathcal N_{K-1}$ all solutions $m$ to 
$$ a_1+a_2\alpha^{-m_2}+\dots+a_{K-1}\alpha^{-m_{K-1}}+a_K\alpha^{-m}=0$$
with $m>m_{K-1}$. If there exists any solution we have found an example satisfying \eqref{eq:not-stable} by multiplying the above equation with $\alpha^m$ and setting
$n_1=m$, $n_2=m-m_2,\dots,n_{K-1}=m-m_{K-1}$ and $n_K=0$.
If no solution exists we compute the set $\mathcal N_K$ consisting of $(K-1)$-tuples of positive integers
as follows. The $(K-1)$-tuple $(m_2,\dots,m_K)$ is contained in $\mathcal N_K$ if and only if $(m_2,\dots,m_{K-1})\in \mathcal N_{K-1}$ and $m_K>m_{K-1}$ is such that
$$\left|a_1+a_2\alpha^{-m_2}+\cdots+a_{K-1}\alpha^{-m_{K-1}}\right|-\frac{A|\alpha|^{-m_K+1}}{|\alpha|-1}<\frac{C_2^{(K-1)}}{2}.$$
Note that $\mathcal N_K$ is a finite set and that $(m_2,\dots,m_K)\not\in \mathcal N_{K}$ implies that
\begin{multline*}
\left|a_1+a_2\alpha^{-m_2}+\cdots+a_k\alpha^{-m_{k}}\right| > \\ \left|a_1+a_2\alpha^{-m_2}+\cdots+a_{K-1}\alpha^{-m_{K-1}}\right| -\frac{A|\alpha|^{-m_K+1}}{|\alpha|-1}\geq  \frac{C_2^{(K-1)}}{2}.
\end{multline*}
for any choice of integers $m_{K+1}< \dots <m_k$ with $m_K<m_{K+1}$. If $\mathcal N_K$ is empty, then we may choose $\mathcal N_i=\emptyset$ for $i=K+1,\dots,k$ and $C_2^{(K)}=\dots=C_2^{(k)}=\frac{C_2^{(K-1)}}2$.
Thus \eqref{eq:lower_bound_1} holds with $C_2=C_2^{(K)}$. In the case that $\mathcal N_K$ is not empty we compute
$$C_2^{(K)}=\min\left\{\frac{C_2^{(K-1)}}2,\min_{(m_2,\dots,m_K)\in\mathcal N_K}\{|a_1+a_2\alpha^{-m_2}+\dots+a_K\alpha^{-m_K}|\}\right\}.$$
Thus for any choice of $n_1>n_2>\dots>n_K\geq 0$ we have  
$$|a_1 \alpha^{n_1}+\dots+a_2 \alpha^{n_2}+\dots+a_K\alpha^{n_K}|>C_2^{(K)} |\alpha|^{n_1}$$
and we completely proved our claim.

If we put $n_i=n_1-m_i$ we obtain from the claim that either there exists some integer $1\leq K\leq k$ and some integers 
 $n_1>\dots>n_K\geq 0$ such that \eqref{eq:not-stable} holds or  \eqref{eq:lower_bound_1} holds with $C_2^{(U)}=C_2^{(k)}$.
Finally note that  \eqref{eq:lower_bound_2} follows from inequality \eqref{eq:lower_bound_1} and applying Lemma \ref{lem:approx}.
\end{proof}

As an immediate corollary of Proposition \ref{prop:stable-growth} we obtain

\begin{corollary}\label{cor:pos-inf}
 Under the assumption that $(a_1,\dots,a_k)$ admits dominance for $(U_n)_{n=0}^{\infty}$ we have that $\mathcal I^{(U)}_{(a_1,\dots,a_k)}>0$.
\end{corollary}

\begin{proof}
Note that $0<C_3\leq \mathcal I^{(U)}_{(a_1,\dots,a_k)}$ by Proposition \ref{prop:stable-growth}.
\end{proof}

Before we end this section we want to draw one more conclusion.

\begin{proposition}\label{prop:n-m-relation}
Under the hypothesis made in Theorem \ref{th:finiteness} there exist positive, effectively computable constants $C_5,C_6,C_7$ and $C_8$ such that
$$C_5<\frac {|\alpha|^{n_1}}{|\beta|^{m_1}}< C_6$$
and
$$C_7<\frac{n_1}{m_1}<C_8$$
provided that $m_1\neq 0$. The constants $C_5,C_6,C_7$ and $C_8$ depend on $(U_n)_{n=0}^\infty$, $(V_m)_{m=0}^\infty$, $\mathcal I^{(U)}_{a_1,\dots,a_k}$, $\mathcal I^{(V)}_{b_1,\dots,b_\ell}$, 
$\mathcal S^{(U)}_{a_1,\dots,a_k}$, $\mathcal S^{(V)}_{b_1,\dots,b_\ell}$, $A$ and $B$ but not on $M=k+\ell$.
\end{proposition}

\begin{proof}
 Combining the results from Corollary \ref{cor:pos-inf} and Lemma \ref{lem:approx} we obtain
 $$C'_5 |\alpha|^{n_1} <C^{(U)}_3|U_{n_1}|<|a_1U_{n_1}+\dots+a_kU_{n_k}|$$
 and
 $$C''_5 |\beta|^{m_1} <C^{(V)}_3|V_{m_1}|<|b_1V_{m_1}+\dots+b_\ell V_{m_\ell}|.$$
 Moreover we have 
 \begin{align*}
  |a_1U_{n_1}+\dots+a_kU_{n_k}|<&\: A(|u|+C_1^{(U)})\left(|\alpha|^{n_1}+\dots+|\alpha|^{n_k}\right)\\
  <&\: A(|u|+C_1^{(U)})\frac{|\alpha|}{|\alpha|-1}|\alpha|^{n_1}\\
  =&\: C_6'' |\alpha|^{n_1}
 \end{align*} 
 and similarly we have that
 $$|b_1V_{m_1}+\dots+b_\ell V_{m_\ell}|<C_6' |\beta|^{m_1}.$$
 Thus we obtain 
 $$C_5=\frac{C''_5}{C''_6}<\frac {|\alpha|^{n_1}}{|\beta|^{m_1}}<\frac{C'_6}{C'_5}=C_6,$$
 where $C'_5,C''_5,C'_6$ and $C''_6$ are effective computable, positive constants.
 
 To obtain the second inequality we take logarithms and obtain
 $$\log C_5 < n_1 \log|\alpha|-m_1\log|\beta|<\log C_6$$
 which yields
 \begin{multline*}
 C_7=-\frac{|\log C_5|}{\log|\alpha|}+\frac{\log|\beta|}{\log|\alpha|}\leq \frac{\log C_5}{m_1 \log|\alpha|}+\frac{\log|\beta|}{\log|\alpha|}<\frac{n_1}{m_1}\\
 <\frac{\log C_6}{m_1 \log|\alpha|}+\frac{\log|\beta|}{\log|\alpha|}\leq \frac{\log C_6}{\log|\alpha|}+\frac{\log|\beta|}{\log|\alpha|}=C_8.
 \end{multline*}
\end{proof}

\begin{remark}
 In case that $(a_1, \dots, a_k)$ comes from a regular digit expansions we have that $\mathcal I_{(a_1,\dots,a_k)}^{(U)}\geq a_1\geq 1$.
 Therefore the statements of Proposition \ref{prop:n-m-relation} hold also in the situation relevant for Theorem~\ref{th:Numeration}.
\end{remark}

\section{Deduction of Theorems \ref{th:finiteness} and \ref{th:Numeration} from Theorem \ref{th:fundamental}}\label{Sec:Deduction}

This section is devoted to the proof of Theorem \ref{th:finiteness} and Theorem \ref{th:Numeration} by applying
Theorem \ref{th:fundamental} with $m=k+\ell+3$. With this choice of $m$ for every pair $(K,L)$ with $K,L\geq 2$ and $K+L=m$
we have that either $K-1\geq k+1$ or $L-1\geq \ell+1$ holds. Thus by the bound \eqref{eq:max} of Theorem \ref{th:fundamental} we get that either
$$n_1<(C \log n_1)^{k+\ell-1}$$
or
$$n_1<C_8m_1<C_8(C \log n_1)^{k+\ell-1}$$
and therefore we have that
$\max\{n_1,m_1\}\leq N$
where $N$ is effectively computable provided that $C$ is. Thus we have proved Theorem \ref{th:finiteness}.

For the proof of Theorem \ref{th:Numeration} we note that $H_G(n)+H_H(n)\leq M$ implies that there are positive integers $k$ and $\ell$ with $k+\ell\leq M$ and positive 
integers $a_1,\dots,a_k$ and $b_1,\dots,b_\ell$ with $\max\{a_1,\dots,a_k,b_1,\dots,b_\ell\}\leq C_{G,H}$ such that
\begin{equation}\label{eq:GH-adic} a_1 G_{n_1}+\dots+a_k G_{n_k}=n=b_1 H_{m_1}+\cdots+b_\ell H_{m_\ell}\end{equation}
for some integers $n_1> \dots > n_k\geq 0$ and $m_1 >\dots >m_\ell\geq 0$. Note that by the premise that $G$ and $H$ are bases for a numeration system and since \eqref{eq:GH-adic} is a
regular digit expansion $C_2,C_3$ do not depend on $M$ by Remark~\ref{rem:lower-bound-numeration}. Moreover, the hypothesis of dominance
are fulfilled, i.e. we may apply Theorem \ref{th:fundamental} with $m=M+3$. Similar as above we obtain
\begin{equation}\label{eq:Th2-bound}
n_1<\max\{1,C_8\}(C \log n_1)^{M-1}=\max\{1,C_8\}(C\log n_1)^{M-1}.
\end{equation}
To solve the inequality above we apply a lemma due to Peth\H{o} and de Weger \cite{Pethoe:1986}. For a proof of Lemma \ref{lem:pdw} we refer to \cite[Appendix B]{Smart:DiGL}.
\begin{lemma} \label{lem:pdw}
Let $u,v \geq 0, h \geq 1$ and $x \in \R$ be the largest solution of $x=u+v(\log{x})^h$. Then
$$
x<\max\{2^h(u^{1/h}+v^{1/h}\log(h^hv))^h, 2^h(u^{1/h}+2e^2)^h\}.
$$
\end{lemma}

Applying Lemma \ref{lem:pdw} with $u=0$, $v=\max\{1,C_8\}C^{M-1}$ and $h=M-1$ yields
\begin{align*}
n_1&<(2C)^{M-1}\max\{1,C_8\} \left(\log\left(\max\{1,C_8\} (C(M-1))^{M-1}\right)\right)^{M-1}\\
&<(C''M\log M)^{M-1} 
\end{align*}
for a suitable constant $C''$. Since by assumption $\mathcal I_{(a_1,\dots,a_k)}^{(U)}\geq a_1\geq 1$ we obtain $\log n <C''' n_1$ for some positive constant $C'''$.
Thus also Theorem \ref{th:Numeration} is proved.

\section{Some auxiliary results}\label{Sec:Aux}

Denote by $\eta_1, \dots, \eta_k$ algebraic numbers, not $0$ or $1$, and by $\log \eta_1, \dots$, $\log \eta_k$
a fixed determination of their logarithms. Let $K=\Q(\eta_1, \dotso, \eta_k)$ and let $d=[K:\Q]$ be the degree of $K$ over $\Q$.
For any $\eta \in K$, suppose that its minimal polynomial over the integers is
\[ g(x) = a_0 x^{\delta} + a_1 x^{\delta - 1} + \dotso + a_{\delta} = a_0 \prod_{j=1}^{\delta} (x - \eta^{(j)})\]
where $\eta^{(j)}$, $j=1, \dotso, \delta$ are the roots of $g(x)$.
The absolute logarithmic Weil height of $\eta$ is defined as
\[ h(\eta) =\dfrac{1}{\delta}\left( \log |a_0| + \sum_{j=1}^{\delta} \log \left( \max \lbrace|\eta^{(j)}|, 1 \rbrace \right)\right). \]
Then the modified height $h'(\alpha)$ is defined by
\[ h'(\eta)=\frac{1}{d}\max \{d h(\eta), |\log\eta|, 1\}.\]
Let us consider the linear form
\[ \L(z_1, \dotso, z_k)=\ell_1 z_1 + \dotso + \ell_k z_k, \]
where $\ell_1, \dots, \ell_k$ are rational integers, not all $0$ and define
\[ h'(\L) = \frac{1}{d}\max\{h(\L),1\}, \]
where $h(\L) = d \log \left(\max_{1 \leq j \leq k} \left\{\frac{|\ell_j|}{\lambda}\right\}\right)$ is the logarithmic Weil height of $\L$,
where $\lambda$ is the greatest common divisor of $\ell_1, \dotso, \ell_k$.
If we write $L=\max\lbrace|\ell_1|, \dotso, |\ell_k|, e \rbrace$, then we get 
\[ h'(\L)\leq \log L. \]

With these notations we are able to state the following result due to Baker and W\"ustholz~\cite{bawu93}.

\begin{theorem}\label{th:BaWu}
If $\Lambda=\L(\log \eta_1, \dots, \log\eta_k) \neq 0$, then
\[\log|\Lambda|\geq -C(k,d)h'(\eta_1)\dotso h'(\eta_k)h'(\L), \]
where
\[ C(k,d)=18(k+1)!k^{k+1}(32d)^{k+2}\log(2kd). \]
\end{theorem}

With $|\Lambda| \leq \frac{1}{2}$, we have $\frac{1}{2}|\Lambda| \leq |\Phi| \leq 2|\Lambda|$, where
\[\Phi = e^{\Lambda}-1 = \eta_1^{\ell_1} \cdots \eta_k^{\ell_k}-1, \]
so that
\[ \log\left|\eta_1^{\ell_1} \cdots \eta_k^{\ell_k}-1\right| \geq \log|\Lambda| - \log 2. \]

Next let us state some known properties of the absolute logarithmic height:
\begin{align*}
 h(\eta \pm \gamma)&\leq h(\eta) + h(\gamma) + \log 2,\\
 h(\eta\gamma^{\pm 1})&\leq h(\eta) + h(\gamma),\\
 h(\eta^n) &= |n|h (\eta), \quad \text{for} \;\; n\in \Z,
 \end{align*}
where $\eta$ and $\gamma$ are some algebraic numbers.
Upon applying Theorem~\ref{th:BaWu}, which is only valid for $\Lambda\neq 0$, we
need to deal with the situation $\Lambda= 0$ separately. We shall apply the following
lemma repeatedly when dealing with this situation (for a proof see \cite[Lemma 1]{Chim:2017}).

\begin{lemma}\label{lem:height}
Let $K$ be a number field and suppose that $\alpha, \beta \in K$ are two algebraic
numbers which are multiplicatively independent. Moreover, let $n, m \in \Z$. Then
there exists an effectively computable constant $C' > 0$ such that
$$h\left(\frac{\alpha^n}{\beta^m}\right)\geq C' \max\{|n|,|m|\}.$$
\end{lemma}

\begin{remark}
 Let us note that instead of Lemma \ref{lem:height} Stewart used in his proof a result due to van der Poorten and Loxton \cite{Poorten:1977,Poorten:1977a} to
 bound certain heights from below in case that some linear form in logarithms vanishes. However, our approach gives usually larger lower bounds especially
 when applied to specific situations than the general result due to van der Poorten and Loxton \cite{Poorten:1977,Poorten:1977a} or recent results due to Vaaler \cite{Vaaler:2014}.
 See also Section \ref{Sec:Examples} where we derive rather good lower bounds by a direct approach.
\end{remark}

\section{Proof of Theorem \ref{th:fundamental}}\label{Sec:Proof}

Before we start with the proof of Theorem \ref{th:fundamental} we introduce some helpful notations.
As already noted, we write $A=\max_{1\leq i \leq k}\{|a_i|\}$ and $B=\max_{1\leq i \leq \ell}\{|b_i|\}$. Further, we write $L=\Q(\alpha,\beta)$ and $D=[L:\Q]$.

As said before we prove Theorem \ref{th:fundamental} by induction on $m$. For the induction base we have to consider the case $m=4$, i.e. $K=L=2$. Since trivially $\max\{n_1-n_1,m_1-m_1\}=0$
the inequalities \eqref{eq:max} and \eqref{eq:seq} are satisfied in this case.
In order to prove inequality \eqref{eq:min} we consider equation~\eqref{eq:main} and collect the ``large terms'' on the left hand side and obtain
\begin{align*}
\left|a_1u\alpha^{n_1}-b_1v\beta^{m_1}\right|\leq & |a_1U_{n_1}-a_1u\alpha^{n_1}|+|b_1V_{m_1}-b_1v\beta^{m_1}|\\
&+|a_2 U_{n_2}|+\dots+|a_kU_{n_k}|+|b_2 V_{m_2}|+\dots+|b_\ell V_{m_\ell}|\\
\leq & |a_1|C_1^{(U)}|\alpha_2|^{n_1}+|b_1|C_1^{(V)}|\beta_2|^{m_1}\\
&+A(|u|+C_1^{(U)})(|\alpha|^{n_2}+\dots+|\alpha|^{n_k})\\
&+B(|v|+C_1^{(V)})(|\beta|^{m_2}+\dots+|\beta|^{m_\ell})\\
\leq & |a_1|C_1^{(U)}|\alpha_2|^{n_1}+|b_1|C_1^{(V)}|\beta_2|^{m_1}\\
&+A(|u|+C_1^{(U)})\frac{|\alpha|}{|\alpha|-1}|\alpha|^{n_2}\\
&+B(|v|+C_1^{(V)})\frac{|\beta|}{|\beta|-1}|\beta|^{m_2}\\
\leq & |a_1|C_1^{(U)}|\alpha_2|^{n_1}+|b_1|C_1^{(V)}|\beta_2|^{m_1}+C_9|\alpha|^{n_2}+C_{10}|\beta|^{m_2}.
\end{align*}
Dividing through $|b_1v\beta^{m_1}|$ we obtain by using the inequality $\frac {|\alpha|^{n_1}}{|\beta|^{m_1}}< C_6$ from Proposition \ref{prop:n-m-relation}
\begin{equation}\label{eq:LinFormIB}
\begin{split}
\left|\frac{a_1u\alpha^{n_1}}{b_1v\beta^{m_1}}-1\right|<& \frac{|a_1|C_1^{(U)}C_6}{|b_1v|}\left|\frac{\alpha}{\alpha_2}\right|^{-n_1}+\frac{|b_1|C_1^{(V)}}{|b_1v|}\left|\frac{\beta}{\beta_2}\right|^{-m_1}\\
&+\frac{C_9C_6}{|b_1v|}|\alpha|^{n_2-n_1}+\frac{C_{10}}{|b_1v|}|\beta|^{m_2-m_1}\\
<&C_{11} \max\{(\alpha')^{n_2-n_1},(\beta')^{m_2-m_1}\}
\end{split}
\end{equation}
where
$$
C_{11}=\max\left\{\frac{(C_9+|a_1|C_1^{(U)})C_6}{|b_1v|},\frac{C_{10}+|b_1|C_1^{(V)}}{|b_1v|}\right\}
$$
and $\alpha'=\min\left\{|\alpha|,\frac{|\alpha|}{|\alpha_2|}\right\}$ and $\beta'=\min\left\{|\beta|,\frac{|\beta|}{|\beta_2}\right\}$.
We consider the linear form 
\begin{equation}\label{eq:LinFormLambdaIB}
 \Lambda=\log \left|\frac{a_1u}{b_1v}\right|+n_1\log|\alpha|-m_1\log|\beta|.
\end{equation}
Assuming that $|\Lambda|>\frac{1}{2}$ or that $\frac{a_1u\alpha^{n_1}}{b_1v\beta^{m_1}}$ is negative would yield
$$\frac{1}{2}<C_{11} \max\{(\alpha')^{n_2-n_1},(\beta')^{m_2-m_1}\}.$$
i.e. $\min\{n_1-n_2,m_1-m_2\}< C_{12}<C_{12}\log n_1$. 

Let us investigate the case that $\Lambda=0$. But, $\Lambda=0$ implies $a_1u\alpha^{n_1}=b_1v\beta^{m_1}$ or equivalently
\begin{equation}\label{eq:Lambda=0-Basis}
 \frac{a_1u}{b_1v}=\frac{\beta^{m_1}}{\alpha^{n_1}}
\end{equation}
and an application of Lemma \ref{lem:height} to equation \eqref{eq:Lambda=0-Basis} yields
$$h\left(\frac{a_1u}{b_1v}\right)=h\left( \frac{\beta^{m_1}}{\alpha^{n_1}}\right)>C' \max\{m_1,n_1\}>n_1C'\max\{1,1/C_8\}$$
which implies 
$$\min\{n_1-n_2,m_1-m_2\}< C_{13}<C_{13}\log n_1$$
with
$$C\geq C_{13}=\frac{h\left(\frac{a_1u}{b_1v}\right)}{C'\max\{1,1/C_8\}}.$$ 

Therefore we may assume that $0<|\Lambda|<\frac{1}{2}$ and we may apply Theorem~\ref{th:BaWu} to \eqref{eq:LinFormLambdaIB} with
\begin{gather*}
 \eta_1= \left|\frac{a_1u}{b_1v}\right|, \qquad \eta_2=|\alpha|, \qquad \eta_3=|\beta|.\\
 \ell_1=1,\qquad \ell_2=n_1,\qquad \ell_3=-m_1,
\end{gather*}
and we obtain
\begin{multline*}
h'(\alpha)h'(\beta)h'\left(\frac{a_1u}{b_1v}\right)C(3,D)\left(\log n_1+\max\{0,-\log C_8\}\right)\\
>\min\{\log \alpha' (n_1-n_2),\log \beta' (m_1-m_2)\}-\log C_{11}-\log 2. 
\end{multline*}
This implies
$$\min\{n_1-n_2,m_1-m_2\}\leq C_{14}\log n_1.$$
Therefore we have established \eqref{eq:min} for the case $m=4$ for any constant $C$ with
$$C\geq\max\left\{C_{12},C_{13},C_{14}\right\}=C_{15}.$$

Next, we have to prove the induction step. Therefore let us assume that \eqref{eq:min}, \eqref{eq:max} and \eqref{eq:seq} hold for a pair $(K',L')$ such that $K'+L'=m-1$ and for a
sequence of pairs $(K'_i,L'_i)$ with $i=4,\dots,m-1$.
We distinguish now three cases:
\begin{description}
 \item[Case 1]  $K'>k+1$;
 \item[Case 2]  $L'>\ell+1$;
 \item[Case 3]  $K'\leq k+1$ and $L'\leq \ell+1$.
\end{description}

\noindent\textbf{Case 1:} If $K'>k+1$ then we have
\begin{align*}
\min\{n_1-n_{K'+1},m_1-m_{L'}\}&=\min\{n_1-n_{K'},m_1-m_{L'}\}\\
&\leq (C\log n_1)^{K'+L'-3}\\
&\leq (C\log n_1)^{K'+L'-2}
\end{align*}
and
\begin{align*}
\max\{n_1-n_{K'},m_1-m_{L'-1}\}&=\max\{n_1-n_{K'-1},m_1-m_{L'-1}\}\\
& \leq (C\log n_1)^{K'+L'-4}\\
& \leq (C\log n_1)^{K'+L'-3},
\end{align*} 
i.e. \eqref{eq:min} and \eqref{eq:max} also hold for the pair $(K,L)=(K'+1,L')$. Moreover, \eqref{eq:seq} holds with $(K_i,L_i)=(K'_i,L'_i)$ for $i=4,\dots,m-1$ and $(K_m,L_m)=(K'+1,L')$.\\

\noindent\textbf{Case 2:} In case that $L'>l+1$ we similarly conclude that \eqref{eq:min}, \eqref{eq:max} and \eqref{eq:seq} hold for the pair $(K,L)=(K',L'+1)$.\\

\noindent\textbf{Case 3:}
We assume that $K'\leq k+1$ and $L'\leq l+1$. However, in view of \eqref{eq:min} we have to distinguish between the following two sub-cases:
\begin{description}
 \item[Case 3A] $n_1-n_{K'}<(C \log n_1)^{K'+L'-3}$ and we choose $(K,L)=(K'+1,L')$,
 \item[Case 3B] $m_1-m_{L'}<(C \log n_1)^{K'+L'-3}$ and we choose $(K,L)=(K',L'+1)$.
\end{description}
Let us note that with this choice inequality \eqref{eq:max} is satisfied in both cases. E.g. assuming that Case 3A holds, we obtain
\begin{align*}
 \max\{n_1-n_{K-1},m_1-m_{L-1}\}& =\max\{n_1-n_{K'},m_1-m_{L'-1}\}\\
 &< (C \log n_1)^{K'+L'-3}\\
& = (C \log n_1)^{K+L-4}.
\end{align*}
But also \eqref{eq:seq} holds with $(K_i,L_i)=(K'_i,L'_i)$ for $i=4,\dots,m-1$ and $(K_m,L_m)=(K'+1,L')$.
Obviously a similar argument also works in the Case 3B.

We are left to prove \eqref{eq:min}. We consider equation \eqref{eq:main} which yields the inequality
\begin{equation}\label{eq:first-ieq-general}
\begin{split}
 \left|a_1u\alpha^{n_1}+\dots \right. &\left. +a_{K-1}u\alpha^{n_{K-1}}-b_1v\beta^{m_1}-\dots-b_{L-1}v\beta^{m_{L-1}}\right|\\
 \leq& |a_1U_{n_1}-a_1u\alpha^{n_1}|+\dots+ |a_{K-1}U_{n_{K-1}}-a_{K-1}u\alpha^{n_{K-1}}|\\
 &+|b_1V_{m_1}-b_1v\beta^{m_1}|+\dots+|b_{L-1}V_{m_{L-1}}-b_{L-1}v\beta^{m_{L-1}}|\\
 &+|a_{K} U_{n_{K}}|+\dots+|a_kU_{n_k}|+|b_{L} V_{m_{L}}|+\dots+|b_\ell V_{m_\ell}|\\
 <& AC_1^{(U)}\left(|\alpha_2|^{n_1}+\dots+|\alpha_2|^{n_{K-1}}\right)\\
 &+BC_1^{(V)}\left(|\beta_2|^{m_1}+\dots+|\beta_2|^{m_{L-1}}\right)\\
 &+A\left(|u|+C_1^{(U)}\right)(|\alpha|^{n_{K}}+\dots+|\alpha|^{n_k})\\
 &+B\left(|v|+C_1^{(V)}\right)(|\beta|^{m_{L}}+\dots+|\beta|^{m_\ell})\\
 <&AC_1^{(U)}\frac{|\alpha_2|}{|\alpha_2|-1}|\alpha_2|^{n_1}+BC_1^{(V)}\frac{|\beta_2|}{|\beta_2|-1}|\beta_2|^{m_1}\\
 &+A\left(|u|+C_1^{(U)}\right)\frac{|\alpha|}{|\alpha|-1}|\alpha|^{n_K}\\
 &+B\left(|v|+C_1^{(V)}\right)\frac{|\beta|}{|\beta|-1}|\beta|^{m_L}\\
 <&C_{16}|\alpha_2|^{n_1}+C_{17}|\beta_2|^{m_1}+C_{18}|\alpha|^{n_K}+C_{19}|\beta|^{m_L},
 \end{split}
\end{equation}
 Now dividing through 
$$|v|\left|b_1\beta^{m_1}+\dots+b_L\beta^{m_{L-1}}\right|>|v|C_2^{(V)}|\beta|^{m_1}>\frac{|v|C_2^{(V)}}{C_6}|\alpha|^{n_1}$$
yields
\begin{equation}\label{eq:lambda-general}
\begin{split}
 \left|\Phi\right|=&\left|\frac{\alpha^{n_{1}}u(a_1+\dots+a_{K-1}\alpha^{n_{K-1}-n_1})}{\beta^{m_1}v(b_1+\dots+b_{L-1}\beta^{m_{L-1}-m_1})}-1\right|\\
 <& \frac{C_{18}C_6}{|v|C_2^{(V)}}|\alpha|^{n_K-n_1}+\frac{C_{19}}{|v|C_2^{(V)}}|\beta|^{m_L-m_1}\\
 &+\frac{C_{16}C_6}{|v|C_2^{(V)}}\left|\frac{\alpha}{|\alpha_2|}\right|^{-n_1}+\frac{C_{17}}{|v|C_2^{(V)}}\left|\frac{\beta}{|\beta_2|}\right|^{-m_1}\\
<&C_{20} \max\{(\alpha')^{n_K-n_1},(\beta')^{m_L-m_1}\},
\end{split} 
\end{equation}
where $\alpha'=\min\left\{|\alpha|,\frac{|\alpha|}{|\alpha_2|}\right\}$ and $\beta'=\min\left\{|\beta|,\frac{|\beta|}{|\beta_2|}\right\}$.

We consider the linear form 
\begin{equation}\label{eq:LinFormLambda}
 \Lambda=\log \left|\frac{u(a_1+\dots+a_{K-1}\alpha^{n_{K-1}-n_1})}{v(b_1+\dots+b_{L-1}\beta^{m_{L-1}-m_1})}\right|+n_1\log|\alpha|-m_1\log|\beta|.
\end{equation}
Assuming that $|\Lambda|>\frac{1}{2}$ or that $\frac{u(a_1+\dots+a_{K-1}\alpha^{n_{K-1}-n_1})}{v(b_1+\dots+b_{L-1}\beta^{m_{L-1}-m_1})}$ is negative would yield
$$\dfrac{1}{2}<C_{20} \max\{(\alpha')^{n_K-n_1},(\beta')^{m_L-m_1}\}$$
and therefore we get $\min\{n_1-n_K,m_1-m_L\}< C_{21}<C_{21} \log n_1$.

Let us investigate the case that $\Lambda=0$. But $\Lambda=0$ implies 
\begin{equation}\label{eq:Lambda=0}
 \frac{\alpha^{n_1}}{\beta^{m_1}}=\frac{u(a_1+\dots+a_{K-1}\alpha^{n_{K-1}-n_1})}{v(b_1+\dots+b_{L-1}\beta^{m_{L-1}-m_1})}.
\end{equation}
However, an application of Lemma \ref{lem:height} to equation \eqref{eq:Lambda=0} yields
\begin{align*}
h\left(\frac{u(a_1+\dots+a_{K-1}\alpha^{n_{K-1}-n_1})}{v(b_1+\dots+b_{L-1}\beta^{m_{L-1}-m_1})}\right)&=h\left( \frac{\alpha^{n_1}}{\beta^{m_1}}\right)\\
&>C' \max\{m_{1},n_{1}\}\\
&>n_1C'\max\{1,1/C_8\}.
\end{align*}
But, on the other hand we have
\begin{equation}\label{eq:height_eta_3}
 \begin{split}
 & h\left(\frac{u(a_1+\dots+a_{K-1}\alpha^{n_{K-1}-n_1})}{v(b_1+\dots+b_{L-1}\beta^{m_{L-1}-m_1})}\right) \\
& \leq h\left(\frac{u}{v}\right)+ h(\alpha)(n_1-n_1)+\dots+h(\alpha)(n_1-n_{K-1})\\
 &\quad \; +h(\beta)(m_1-m_1)+\dots+h(\beta)(m_1-m_{L-1})\\
 &\quad \; + (K-1)\log A+(L-1)\log B+\log(K-1)+\log(L-1) \\
& \leq h\left(\frac{u}{v}\right)+m(\log A+\log B+1)\\
 &\quad \; + \max\{h(\alpha),h(\beta)\}\left(1+(C \log n_1)+\dots+(C \log n_1)^{m-4}\right)\\
 & \leq h\left(\frac{u}{v}\right) +(\log A+\log B+1)(C \log n_1)^{m-4}\\
 &\quad \;+ \max\{h(\alpha),h(\beta)\}\overbrace{\frac{C\log n_1}{C\log n_1 -1}}^{<1.01}(C \log n_1)^{m-4}\\
& \leq    C_{22} (C \log n_1)^{m-4}.
\end{split}
\end{equation}
Note that $\frac{C\log n_1}{C\log n_1 -1}<1.01$ holds provided that $C>92$ which we clearly may assume. Furthermore, note that each exponent $n_1-n_{K'}$ or $m_1-m_{L'}$
is equal to $\max\{n_1-n_{K_i},m_1-m_{L_i}\}$ for some $4\leq i \leq m-1$. Thus we obtain 
$$ C_{22} (C \log n_1)^{m-4}> n_1C'\max\{1,1/C_8\}$$
and
\begin{align*}
\min\{n_1-n_K,m_1-m_L\}&< \frac{C_{22}}{C'\max\{1,1/C_8\}} (C \log n_1)^{m-4}\\
& \leq (C \log n_1)^{m-3}
\end{align*}
provided that
$$C \geq\frac{C_{22}}{C'\max\{1,1/C_8\}}.$$ 

By the computations in the paragraph above we may assume that $0<|\Lambda|<\frac{1}{2}$ and we may apply Theorem~\ref{th:BaWu} to \eqref{eq:LinFormLambda} with
\begin{gather*}
 \eta_1=  \left|\frac{u(a_1+\dots+a_{K-1}\alpha^{n_{K-1}-n_1})}{v(b_1+\dots+b_{L-1}\beta^{m_{L-1}-m_1})}\right| , \qquad \eta_2= |\alpha| , \qquad \eta_3= |\beta| \\
 \ell_1=1,\qquad \ell_2=n_1,\qquad \ell_3=-m_1
\end{gather*}
From our previous height estimates \eqref{eq:height_eta_3} we know that
$$h(\eta_1)\leq C_{22}(C\log n_1)^{m-4}.$$
Moreover, we know that
$$\max\{n_{1},m_{1}\}\leq \max\{1,1/C_8\} n_1.$$
Thus we obtain from Theorem \ref{th:BaWu} together with inequality \eqref{eq:lambda-general} 
\begin{multline*}
 C(3,D)h'(\alpha)h'(\beta)C_{22}(C\log n_1)^{m-4}\left( \log n_1 + \max\{0,\log (1/C_8)\} \right)\\>\min\{\log \alpha'(n_1-n_K),\log\beta'(m_1-m_L)\} -\log C_{20}-\log 2
\end{multline*}
Let us write $C_{23}=C(3,D)h'(\alpha)h'(\beta)C_{22}$, then we get
\begin{align*}
 \min&\{n_1-n_K,m_1-m_L\}\\
 &<\frac{\log (2C_{20})+ C_{23}\left(1+ \frac{\max\{0,\log (1/C_8)\}}{\log n_1} \right) C_{22}C^{m-4}(\log n_1)^{m-3}}{\min\{\log \alpha',\log \beta'\}} \\
 & \leq (C\log n_1)^{m-3} 
\end{align*}
provided that 
\begin{align*}
C &\geq \frac{\frac{\log (2C_{20})}{(C\log n_1)^{m-3}} + C_{23}\left(1+ \frac{\max\{0,\log (1/C_8)\}}{\log n_1} \right) C_{22}C}{\min\{\log \alpha',\log \beta'\}} \\
&\geq \frac{\frac{\log (2C_{20})}{(C\log 3)^{m-3}} + C_{23}\left(1+ \frac{\max\{0,\log (1/C_8)\}}{\log 3} \right) C_{22}C}{\min\{\log \alpha',\log \beta'\}} \\
 &=: C_{24}
\end{align*}
Note that we assume that $3\leq n_1$.
In particular \eqref{eq:min}, \eqref{eq:max} and \eqref{eq:seq} is satisfied if we have chosen $C$ large enough, i.e.
$$C\geq \max\left\{C_{24},\frac{C_{22}}{C'\max\{1,1/C_8\}},C_{21},C_{15}\right\}.$$

\section{Practical Implementation}\label{Sec:Implementation}

Usually the use of Baker's method provides an upper bound $N$ for $\max\{n_1,m_1\}$ which is very large.
Even if we take more care in computing the upper bounds in the course of the proof of Theorem \ref{th:fundamental}
the upper bounds are way too large for using a simple search algorithm to enumerate all solutions to \eqref{eq:main}. Even using the best known results for linear forms
in three logarithms (e.g. the bounds given in \cite{Bugeaud:2006a,Bugeaud:2006, Mignotte:kit}) the upper bounds would still be too large.

However if $k+\ell$ stays reasonably small, say $k+\ell\leq 5$ or maybe $6$, it is still possible to use the reduction method due to Baker and Davenport \cite{Baker:1969} in combination with Legendre's
theorem on continued fractions to obtain a much smaller upper bound for $\max\{n_1,m_1\}$. The reduction process can be done inductively. This section is twofold.
Firstly we want to describe the inductive way to use the above mentioned reduction methods. Secondly we want to give a short account on how to use these methods. In the next section 
we utilize these methods by an example and prove Theorem \ref{th:Zeckendorf-Binary}.

Let us present our reduction procedure. Therefore we assume that we are given an inequality of the form
\begin{equation}\label{eq:BD-ieq}
 |n \lambda -m \kappa +\nu|<AB^{-k}
\end{equation}
with $\lambda,\kappa,\nu\in \R^*$, $n,m,k \in \Z$, $n\leq M$ and $A,B\in \R$ such that $A>0$ and $B>1$. Let us denote by $\|x\| = \min \{ |x-n| : n \in \Z \}$
the distance from $x$ to the nearest integer. Let us note that the inequalities \eqref{eq:LinFormIB} and \eqref{eq:lambda-general} are of the form \eqref{eq:BD-ieq}.
Assume that the upper bound for $\max\{n_1,m_1\}$ obtained by Theorem \ref{th:finiteness} is $N$. 
Thus in many cases the following lemma can be applied:

\begin{lemma}[Bravo et. al. \cite{Bravo:2017a}] \label{lem:Reduction}
Let $M$ be a positive integer, let $p/q$ be a convergent of the continued fraction of the irrational $\tau$ such that $q > 6M$,
and let $A, B, \mu$ be some real numbers with $A > 0$ and $B > 1$. Let $\varepsilon := \|\mu q\| - M \| \tau q\|$.
If $\varepsilon > 0$, then there is no solution to the inequality
\[ 0 < \left| n \tau - m + \mu \right| < AB^{-k}, \]
in positive integers $n, m$ and $k$ with
\[ n \leq M \quad \text{and} \quad k \geq \dfrac{\log (Aq/ \varepsilon )}{\log B}. \]
\end{lemma}

Typically Lemma \ref{lem:Reduction} is applied to \eqref{eq:BD-ieq} after dividing through $\kappa$ and choosing $M=N$.
We take the smallest denominator $q$ of a convergent $\frac{p}{q}$ to $\tau=\frac{\lambda}{\kappa}$ such that $q>6M$ and test whether $\varepsilon>0$. If $\varepsilon>0$
we have a new, hopefully much smaller upper bound for $k$. In case that we get $\varepsilon<0$ we test whether $\varepsilon>0$ for the next larger denominator $q$ and so on. Sometimes it happens
that we do not find a denominator $q$ such that $\varepsilon>0$. Typically that happens if $\lambda,\kappa$ and $\nu$ are linearly dependent over $\Q$. In this case a good
approximation to $\tau$ is also a good approximation to $\mu$ and $\varepsilon$ will stay negative for every $q$. But in the case that $\lambda,\kappa$ and $\nu$ satisfy some
linear relation
$$ a\lambda+b\kappa+c\nu=0$$
over $\Q$ equation \eqref{eq:BD-ieq} turns into the form
\begin{equation}\label{eq:cont-frac}
 \left|\tau -\frac{m'}{n'}\right|<\frac{A'}{n'}B^{-k}
\end{equation}
with $\tau= \frac{\lambda}{\kappa}$ and $m'=mc+b$, $n'=nc-a$ and $A'=\frac{Ac}{\kappa} $. But, if $\frac{A'}{n'}B^{-k}<\frac{1}{2{n'}^2}$, i.e. if $k>\frac{\log (2N'A')}{\log B}$ 
with $N'=Nc-a$, then by a criterion due to Legendre $\frac{m'}{n'}$ is a convergent to $\tau$ of the form $\frac{p_j}{q_j}$ for some $j=0, 1, 2, \dots, J$, with
$$J=\min\{j\: :\: q_{j+1} > N'\}.$$
However it is well known (see e.g. \cite[page 47]{Baker:NT}) that 
$$\frac{1}{(a_{j+1} + 2)q_j^2} < \left|\tau - \frac{p_j}{q_j} \right|,$$
where $[a_0,a_1,\dots]$ is the continued fraction expansion of $\tau$. Let us write $S=\max \lbrace a_{j+1}: j=0, 1, 2, \dots, J \rbrace$, then we have
$$\frac{1}{ (S+2) q_j^2} < \frac{A'}{q_j} B^{-k} $$
and $q_j$ divides $n'$. Thus we get the inequality
$$ A'(S+2) q_j > B^{k}$$
which yields $k < \frac{\log((S+2) q_j A')}{\log B}$, a new upper bound which we expect to be of the size of the bound obtained by the Baker-Davenport reduction method (cf. Lemma \ref{lem:Reduction}).

Our next goal is to describe how we can apply the technique described above to our specific problem. 
Assume we are given an upper bound $N$ for $n_1$ and $m_1$, e.g. the bound we obtain in the proof of Theorem \ref{th:finiteness}. Then we consider in a first step the linear form
\eqref{eq:LinFormIB} given in the form 
\begin{equation}\label{eq:LinFormIB-Red}
\left|n_1\frac{\log \alpha}{\log \beta}-m_1+\frac{\log(\frac{a_1u}{b_1v})}{\log \beta}\right|<\frac{2C_{11}}{\log \beta} \max\{(\alpha')^{n_2-n_1},(\beta')^{m_2-m_1}\}
\end{equation}
and use Lemma \ref{lem:Reduction}. In case that $1,\frac{\log \alpha}{\log \beta}$ and $\frac{\log(\frac{a_1u}{b_1v})}{\log \beta}$ are linearly dependent
we use continued fractions directly. Thus we hopefully obtain a rather small new upper bound $B_{2,2}$ for $\min\{n_1-n_2,m_1-m_2\}$. 

Let us assume that we have found (small) upper bounds for a pair $(K',L')$ with $K',L'\geq 2$ and $K'+L'=M$ such that 
\begin{equation*}
\begin{split}
&n_1-n_{K'-1}<B^{(n)}_{K'-1}, \quad m_1-m_{L'-1}<B^{(m)}_{L'-1},\\
&\min\{n_1-n_{K'},m_1-m_{L'}\}\leq B_{K',L'}.
\end{split}
\end{equation*}
Then we can hopefully find rather small new upper bounds for a pair $(K,L)$ with $K,L\geq 2$ and $K+L=M+1$ such that 
\begin{equation}\label{eq:Baker-Davenport-Induction}
\begin{split}
& n_1-n_{K-1}<B^{(n)}_{K-1}, \quad m_1-m_{L-1}<B^{(m)}_L,\\
&\min\{n_1-n_K,m_1-m_L\}\leq B_{K,L}.
 \end{split}
\end{equation}
Note that if $K+l\geq m+3$ we have an upper bound for $\max\{n_1,m_1\}$ (cf. Section \ref{Sec:Deduction}).
To prove this claim we proceed as in the proof of Theorem \ref{th:fundamental}. Thus we distinguish between the three cases
\begin{description}
 \item[Case 1]  $K'>k+1$,
 \item[Case 2]  $L'>\ell+1$,
 \item[Case 3]  $K'\leq k+1$ and $L'\leq \ell+1$.
\end{description}

\noindent\textbf{Case 1:} If $K'>k+1$ then we have
$$
\min\{n_1-n_{K'+1},m_1-m_{L'}\}=\min\{n_1-n_{K'},m_1-m_{L'}\}\leq B_{K',L'}
$$
and
$$n_1-n_{K'}=n_1-n_{K'-1}\leq B^{(n)}_{K'-1}, \quad m_1-m_{L'-1}\leq B^{(m)}_{L'-1}$$
i.e. we have found small upper bounds in case that $(K,L)=(K'+1,L')$.\\

\noindent\textbf{Case 2:} In case that $L'>\ell+1$ we similarly find small upper bounds for the pair $(K,L)=(K',L'+1)$.\\

\noindent\textbf{Case 3:}
We assume that $K'\leq k+1$ and $L'\leq \ell+1$. However in view of~\eqref{eq:Baker-Davenport-Induction} we have to distinguish between the following two sub-cases:
\begin{description}
 \item[Case 3A] $n_1-n_{K'}<B_{K',L'}$ and we choose $(K,L)=(K'+1,L')$,
 \item[Case 3B] $m_1-m_{L'}<B_{K',L'}$ and we choose $(K,L)=(K',L'+1)$.
\end{description}
In case 3A we get
$$n_1-n_{K-1}=n_1-n_{K'}\leq B_{K',L'}, \quad m_1-m_{L-1}=m_1-m_{L'-1}\leq B^{(m)}_{L'-1}$$
and in case 3B we get
$$n_1-n_{K-1}=n_1-n_{K'-1}\leq B^{(n)}_{K'-1}, \quad m_1-m_{L-1}=m_1-m_{L'}\leq B_{K',L'}.$$
Moreover we consider inequality \eqref{eq:first-ieq-general} and obtain
\begin{equation}\label{eq:Baker-Davenport-step}
 \left|\frac{\log \left|\frac{u(a_1+\dots+a_{K-1}\alpha^{n_{K-1}-n_1})}{v(b_1+\dots+b_{L-1}\beta^{m_{L-1}-m_1})}\right|}{\log|\beta|}+n_1\frac{\log|\alpha|}{\log|\beta|}-m_1\right|
 < 2C_{20} \max\{(\alpha')^{n_K-n_1},(\beta')^{m_L-m_1}\},
\end{equation}
provided that $|\Lambda|<0.5$ defined as in \eqref{eq:LinFormLambda}. But as already explained in the proof of Theorem~\ref{th:fundamental} the inequality $|\Lambda|\geq 0.5$ would yield
 $\min\{n_1-n_K,m_1-m_L\}< C_{21}$, where $C_{21}$ is typically very small. We apply for all possible values of $0<n_1-n_2<\dots<n_1-n_{K-1}<B^{(n)}_K$ and  $0<m_1-m_2<\dots<m_1-m_{L-1}<B^{(m)}_L$
 Lemma \ref{lem:Reduction} or in case of some linear dependency the continued fraction method to obtain an upper bound
 $$\min\{n_1-n_K,m_1-m_L\}\leq B_{K,L}.$$ 
 
 In practice we can proceed as follows. First we compute $B_2$ such that $\min\{n_1-n_2,m_1-m_2\}\leq B_2=B_{2,2}$. Assume we have computed a bound $B_M$ such that for all pairs $(K',L')$
 with $K',L'\geq 2$ and $K'+L'=M$ we have
 $$n_1-n_{K'-1},m_1-m_{L'-1},\min\{n_1-n_{K'},m_1-m_{L'}\}<B_M$$
 we can compute an upper bound $B_{M+1}$ for all pairs $(K,L)$ with $K,L\geq 2$ and $K+L=M+1$ with
 $$n_1-n_{K-1},m_1-m_{L-1},\min\{n_1-n_{K},m_1-m_{L}\}<B_{M+1}$$
 by applying our reduction method to all inequalities \eqref{eq:Baker-Davenport-step} for each such pair $(K,L)$. Note that $B_{k+\ell+3}$ will yield an upper bound for $\max\{n_1,m_1\}$. 
 
\section{An Example}\label{Sec:Examples}

This section is devoted to prove Theorem \ref{th:Zeckendorf-Binary}. We start with some simple observations. Let us remind the Binet formula, that is
$$F_n=\frac{\alpha^{n+2}-\beta^{n+2}}{\sqrt{5}},$$
with $\alpha=\frac{1+\sqrt{5}}2$ and $\beta=\frac{1-\sqrt{5}}2$. Note that in view of digit expansions (see section \ref{Sec:Notations}) we shifted the index by two 
with respect to the more common definition of the Fibonacci sequence starting with $F_0=0$ and $F_1=1$. Note that with our definition we have that $\alpha^n\leq F_n< \frac{3}{\sqrt 5}\alpha^n$ for all $n\geq 0$. 

Moreover, we want to emphasize that the regularity condition \eqref{eq:regular} in 
case of the Fibonacci sequence means that no two consecutive Fibonacci numbers may appear in the digit expansion. In other words the case that $\epsilon_k=\epsilon_{k+1}=1$ must not
appear in the digit expansion \eqref{eq:repr}.

Now let $n$ be an integer such that
\begin{equation}\label{eq:Zeckendorf-binary}
n=F_{n_1}+\dots+F_{n_k}=2^{m_1}+\dots+2^{m_\ell}
\end{equation}
is its Zeckendorf and binary expansion respectively.
For technical reasons let us assume that $n_1>100$. Then we obtain that
$$ \alpha^{n_1}\leq F_{n_1}\leq n=F_{n_1}+\dots+F_{n_k}<F_{n_1+1}\leq \frac{3}{\sqrt 5}\alpha^{n_1+1}$$
where the second inequality holds since we assume that $n=F_{n_1}+\dots+F_{n_k}$ is a regular expansion. We also have
$$2^{m_1}\leq n=2^{m_1}+\dots+2^{m_\ell}<2\cdot 2^{m_1}.$$
These two inequalities together yield
\begin{equation}\label{eq:ZB-m1-n1-ieq}
 2^{m_1}<\frac{3\alpha}{\sqrt 5}\alpha^{n_1}\quad \text{and}\quad \alpha^{n_1}<2\cdot 2^{m_1}.
\end{equation}

In order to prove the first part of Theorem \ref{th:Zeckendorf-Binary} we prove the following proposition which can be seen as a special case of Theorem \ref{th:fundamental}.

\begin{proposition}\label{prop:Zeckendorf-binary}
If $H_Z(n)+H_b(n)\leq M$, then for every $m\geq 4$ there exists a pair $(K,L)$ 
of positive integers $K,L\geq 2$ with $K+L=m$ such that 
$$
 \min\{(n_1-n_K)\log \alpha, (m_1-m_L)\log 2\}\leq (C\log n_1)^{m-3}
$$
and
$$
 \max\{(n_1-n_{K-1})\log \alpha,(m_1-m_{L-1}\log 2\}\leq (C \log n_1)^{m-4}
$$
Moreover, to the pair $(K,L)$ there exists a sequence of pairs $(K_i,L_i)$ with $i=4,\dots, m$ such that $K_i+L_i=i$, $2\leq K_1\leq \dots \leq K_m=K$ and $2\leq L_1\leq \dots \leq L_m=L$
such that
$$
 \max\{(n_1-n_{K_i})\log \alpha,(m_1-m_{L_i})\log 2\}\leq (C \log n_1)^{i-4} \quad 4\leq i\leq m-1.
$$
We can choose $C=4.17 \cdot 10^{13}$ provided that $n_1>100$.
\end{proposition}

\begin{proof}
 As in the proof of Theorem \ref{th:fundamental} we proceed by induction on $m$. Therefore we start with the case $m=4$, i.e. $K=L=2$. We collect the large terms occurring in \eqref{eq:Zeckendorf-binary},
 namely $2^{m_1}$ and $\alpha^{n_1}$, on the left hand side and obtain
 \begin{align*}
 \left|\frac{\alpha^{n_1+2}}{\sqrt 5}-2^{m_1}\right|
  & \leq \frac{|\beta|^{n_1+2}}{\sqrt 5}+\left|F_{n_2}+\dots+F_{n_k}-2^{m_2}-\dots-2^{m_\ell}\right|\\
 & \leq \frac{1}{\sqrt{5}}+\max\{F_{n_2+1},2^{m_2+1}\}\\
 & \leq \max\left\{\alpha^{n_2+2},\left(2+\frac{1}{\sqrt 5}\right)2^{m_2}\right\}.
 \end{align*}
 We divide through $2^{m_1}$ and obtain
\begin{equation}\label{eq:Zeckendorf-IB-ieq}
 |\Phi|=\left|\alpha^{n_1}2^{-m_1}\frac{\alpha^2}{\sqrt 5}-1\right|\leq \max\left\{2\alpha^2 \alpha^{n_2-n_1},\left(2+\frac{1}{\sqrt 5}\right)2^{m_2-m_1}\right\}.
\end{equation}
  Since $2,\alpha$ and $\sqrt{5}$ are multiplicatively independent the left hand side cannot vanish. We consider the linear form in logarithms
 $$
 \Lambda=n_1\log\alpha-m_1\log 2+\log \frac{\alpha^2}{\sqrt 5}.
 $$
 Assume for the moment that $|\Lambda|<0.5$ then we obtain by using Theorem \ref{th:BaWu} that 
 $$4.161 \cdot 10^{13} \log n_1>-2\log(2\alpha)+\min\{(n_1-n_2)\log \alpha ,(m_1-m_2)\log 2\}.$$
 Let us note that $h'(\alpha)=1/2$, $h'(2)=\log 2$ and 
 $$h'\left(\frac{\alpha^2}{\sqrt 5}\right)=\log \alpha+\frac 12 \log 5.$$
 Thus we obtain the proposition in case that $m=4$ with $C=4.17 \cdot 10^{13}$. Note that in the case that $|\Lambda|\geq 0.5$ the bounds for $n_1-n_2$ and $m_1-m_2$ are
 rather small and are covered by the bounds found in the case that $|\Lambda|<0.5$.
 
 Let us assume that the proposition holds for a pair $(K',L')$ with $K'+L'=m-1$. In case that $K'>k+1$ or $L'>\ell+1$ we can similarly conclude as in the proof of Theorem
 \ref{th:fundamental} that the proposition holds for the pair $(K,L)=(K'+1,L')$ and $(K,L)=(K',L'+1)$ respectively. Therefore we may assume that $K'\leq k+1$ and $L'\leq \ell+1$.
 
 If we assume that
 $$ \min\{(n_1-n_{K'})\log \alpha ,(m_1-m_{L'})\log 2 \}\leq (C\log n_1)^{m-3}=(n_1-n_{K'})\log \alpha $$
 we put $(K,L)=(K'+1,L')$ and in case that
 $$ \min\{(n_1-n_{K'})\log \alpha ,(m_1-m_{L'})\log 2 \}\leq (C\log n_1)^{m-3}=(m_1-m_{L'})\log 2 $$
 we put $(K,L)=(K'+1,L')$. With this notation we collect the large terms on the left hand side and obtain similarly as in the case that $K=L=2$ the following inequality:
 \begin{align*}
 &\left|\frac{\alpha^{n_1+2}+\dots+\alpha^{n_{K-1}+2}}{\sqrt 5}-(2^{m_1}+\dots+2^{m_{L-1}})\right|\\
 &\qquad \leq\frac{|\beta|^{n_1+2}+\dots+|\beta|^{n_{K-1}+2}}{\sqrt 5}+\left|F_{n_K}+\dots+F_{n_k}-2^{m_L}-\dots-2^{m_\ell}\right|\\
 &\qquad\leq \frac{1}{\sqrt{5}}+\max\{F_{n_K+1},2^{m_L+1}\}\\
 &\qquad \leq  \max\left\{\alpha^{n_K+2},\left(2+\frac{1}{\sqrt 5}\right)2^{m_L}\right\}.
 \end{align*}
Note that 
 $$|\beta|^{n_1+2}+\dots+|\beta|^{n_{K-1}+2}<  |\beta|^2 (1+|\beta|+|\beta|^2+ \dots)=\frac{|\beta|^2}{1-|\beta|}=1.$$
 We divide through  $2^{m_1}+\dots+2^{m_{L-1}}>2^{m_1}$ and obtain
 \begin{equation}\label{eq:Zeckendorf-IS-ieq}
 \begin{split}
 |\Phi|& =\left|\alpha^{n_1}2^{-m_1}\frac{\alpha^2(1+\alpha^{n_2-n_1}+\dots+\alpha^{n_{K-1}-n_1})}{\sqrt 5 (1+2^{m_2-m_1}+\dots+2^{m_{L-1}-m_1})}-1\right|\\
& \leq  \max\left\{2\alpha^2 \alpha^{n_K-n_1},\left(2+\frac{1}{\sqrt 5}\right)2^{m_L-m_1}\right\}.
 \end{split}
 \end{equation}
 Therefore we consider the following linear form in logarithms:
  $$
 \Lambda=n_1\log\alpha-m_1\log 2+\log \overbrace{\frac{\alpha^2}{\sqrt 5}\frac{1+\alpha^{n_2-n_1}+\dots+\alpha^{n_{K-1}-n_1}}{1+2^{m_2-m_1}+\dots+2^{m_{L-1}-m_1}}}^{:=\gamma_3}.
 $$
 Let us estimate the height of $\gamma_3$. We obtain by using a similar calculation as in \eqref{eq:height_eta_3} the upper bound
 \begin{align*}
  h(\gamma_3)&\leq h\left(\frac{\alpha^2}{\sqrt{5}}\right)+h(1+\alpha^{n_2-n_1}+\dots+\alpha^{n_{K-1}-n_1})\\
  &\quad+h(1+2^{m_2-m_1}+\dots+2^{m_{L-1}-m_1})\\
 & \leq  1.3+h(1)+h(\alpha^{n_1-n_2})+\dots+h(\alpha^{n_1-n_{K-1}})+\log K\\
  &\quad +h(1)+h(2^{m_1-m_2})+\dots+h(2^{m_1-m_{L-1}})+\log L\\
  &\leq  1.3+2\log m +h(\alpha)(n_1-n_2)+\dots+h(\alpha)(n_1-n_{K-1})\\
  &\quad+ h(2)(m_1-m_2)+\dots+h(2)(m_1-m_{L-1})\\
  & \leq  1.3+2\log m + C\log n_1+\dots+(C\log n_1)^{m-4}\\
 & \leq  1.3+2\log m + \frac{C \log n_1}{C \log n_1-1}(C \log n_1)^{m-4}\\
& < (1.0001 C \log n_1)^{m-4}\\
 \end{align*}  
 provided that $C>10^{13}$ which we clearly may assume.
 Let us assume for the moment that $\Lambda\neq 0$ and $|\Lambda|<0.5$, then we get by Theorem \ref{th:BaWu}
 $$
 3.239 \cdot 10^{13} (C \log n_1)^{m-4} \log n_1 > -2\log(2\alpha)+\min\{(n_1-n_K)\log \alpha,(m_1-m_L)\log 2\}
 $$
 hence
 $$  \min\{(n_1-n_K)\log \alpha,(m_1-m_L)\log 2\}<3.24 \cdot 10^{13} C^{m-4} (\log n_1)^{m-3} \leq (C \log n_1)^{m-3}$$
 and we obtain the proposition with $C=4.17 \cdot 10^{13}$.
 
 In case that $|\Lambda|\geq 0.5$ the bounds for $\min\{ (n_1-n_K)\log \alpha, (m_1-m_L)\log 2\}$ are small and covered by the bounds given in the proposition. Therefore we are
 left with the case that $\Lambda=0$. We want to find a lower bound for $h(2^{m_1}/\alpha^{n_1})$. Therefore we distinguish between the following two cases:
 \begin{itemize}
  \item $2^{m_1}>\alpha^{n_1}$;
  \item $2^{m_1}\leq\alpha^{n_1}$.
 \end{itemize}
 In the first case we find
 \begin{align*}
  2 h(2^{m_1}/\alpha^{n_1}) &= \log \max\left\{\left|\frac{2^{m_1}}{\alpha^{n_1}}\right|,1\right\}+\log \max\left\{\left|\frac{2^{m_1}}{\beta^{n_1}}\right|,1\right\}\\
  &=\log \left|\frac{2^{m_1}}{\alpha^{n_1}}\right|+\log \left|\frac{2^{m_1}}{\beta^{n_1}}\right|\\
 & =m_1\log 2 -n_1\log |\alpha|+m_1\log 2-n_1\log |\beta|\\
&  =2m_1\log 2\\
&>-2\log 2+ 2 n_1 \log \alpha
 \end{align*} 
 and in the second case we find due to \eqref{eq:ZB-m1-n1-ieq} that
 $$ 2h\left(\frac{2^{m_1}}{\alpha^{n_1}}\right)= \log \left|\frac{2^{m_1}}{\beta^{n_1}}\right|=m_1\log 2+n_1 \log \alpha>-\log 2+ 2 n_1 \log \alpha.$$
 Since $\Lambda=0$ we have that
 $$h\left(\frac{2^{m_1}}{\alpha^{n_1}}\right)=h(\gamma_3).$$
 Comparing the lower bound for $h(2^{m_1}/\alpha^{n_1})$ with the upper bound for $h(\gamma_3)$ we obtain
 $$ n_1 \log \alpha-\log 2<(1.0001 C \log n_1)^{m-4}$$
 and therefore an absolute bound for $n_1$. However, we also find
 $$(n_1-n_K)\log \alpha\leq n_1 \log \alpha< \log 2 +(1.0001 C \log n_1)^{m-4}< (C \log n_1)^{m-3}$$
 for $C=4.17 \cdot 10^{13}$. Therefore the proof of the proposition is complete.
\end{proof}

In order to proof the first part of Theorem \ref{th:Zeckendorf-Binary} we choose $m=M+3$ and obtain that either $K-1\geq k+1$ or $L-1\geq \ell +1$. Thus either
$$n_1 \log \alpha < (4.17 \cdot 10^{13} \log n_1)^{M-1}$$
or 
$$n_1\log \alpha-\log 2 <m_1 \log 2 < (4.17 \cdot 10^{13} \log n_1)^{M-1}$$
in any case we obtain that
$$n_1 < \frac{(4.18 \cdot 10^{13} \log n_1)^{M-1}}{\log \alpha}$$
and by an application of Lemma \ref{lem:pdw} we obtain
\begin{equation}\label{eq:ZB-upper}
 \begin{split}
 n_1& < \frac 1{\log \alpha}\left(8.36 \cdot 10^{13} \times \phantom{\frac 1{\log \alpha}}\right.\\
 &\qquad\quad \left.\left(\log\left(\frac 1{\log \alpha}\right)+(M-1)(\log (M-1) +\log(4.18\cdot 10^{13}))\right)\right)^{M-1}\\
& <\frac 1{\log \alpha}\left(8.36 \cdot 10^{13} (M-1)\times \phantom{\frac 1{\log \alpha}}\right.\\
&\qquad\quad \left.\left(\log\left(\frac 1{\log \alpha}\right)+\log (M-1) +31.364)\right)\right)^{M-1}\\
& <\frac 1{\log \alpha}\left(3.96 \cdot 10^{15} M \log M\right)^{M-1}.
 \end{split}
\end{equation}
Thus we get
$$n_1<\left(8.23 \cdot 10^{15} M\log M\right)^{M-1}$$
and therefore we have proved the first part of Theorem \ref{th:Zeckendorf-Binary}.

In order to prove the second part of Theorem \ref{th:Zeckendorf-Binary} we have to consider several Diophantine equations. Most of these Diophantine equations have already been solved
in the past. Let us start by considering the case $M=2$. This case is basically the Diophantine equation $F_n=2^m$. Since Bugeaud et.al. \cite{Bugeaud:2006} we know that the only perfect powers in the 
Fibonacci sequence are $1,8$ and $144$ and therefore this case is solved.

In the case that $M=3$ we have to distinguish between two cases. We have to consider the Diophantine equations $F_n=2^{m_1}+2^{m_2}$ and $F_{n_1}+F_{n_2}=2^m$. The second equation was solved by
Bravo and Luca \cite{Bravo:2016} and the solution of the first equation is contained in Theorem 2.2 of \cite{Bugeaud:2013}, where all Fibonacci numbers are determined with at most four binary digits.

In the case that $M=4$ we have to distinguish between the three equations $F_n=2^{m_1}+2^{m_2}+2^{m_3}$ solved by Bugeaud et.al. \cite[Theorem 2.2]{Bugeaud:2013} $F_{n_1}+F_{n_2}=2^{m_1}+2^{m_2}$ solved by 
Chim and Ziegler \cite{Chim:2017b} and $F_{n_1}+F_{n_2}+F_{n_3}=2^{m}$ (see Theorem \ref{th:reduction} below).

In the case that $M=5$ we have to distinguish between the four equations $F_n=2^{m_1}+2^{m_2}+2^{m_3}+2^{m_4}$ solved by Bugeaud et.al. \cite[Theorem 2.2]{Bugeaud:2013} $F_{n_1}+F_{n_2}=2^{m_1}+2^{m_2}+2^{m_3}$ and
$F_{n_1}+F_{n_2}+F_{n_3}=2^{m_1}+2^{m_2}$ both solved by Chim and Ziegler \cite{Chim:2017b} and $F_{n_1}+F_{n_2}+F_{n_3}+F_{n_4}=2^{m}$ (see Theorem \ref{th:reduction} below).

Therefore we are left to prove

\begin{theorem}\label{th:reduction}
 The only powers of two that have a Zeckendorf expansion with exactly three non-zero digits are $32,64,128,256$ and $1024$.
 Moreover, there are no powers of two that have a Zeckendorf expansion with exactly four non-zero digits.
\end{theorem}

\begin{proof}
 We consider the Diophantine equation
 \begin{equation}\label{eq:Power2Zeckendorf}
  F_{n_1}+F_{n_2}+F_{n_3}+F_{n_4}=2^m
 \end{equation}
 where $n_1-1>n_2$, $n_2-1>n_3$, $n_3-1>n_4$ and $n_4\geq 0$ or $n_4=-2$. With this restriction we impose the condition that the right hand side is a Zeckendorf expansion with three or four non-zero digits.
 
 Due to the first part of Theorem \ref{th:Zeckendorf-Binary} and more specific due to \eqref{eq:ZB-upper} we know that $n_1<3.1 \cdot 10^{64}$ and from \eqref{eq:Zeckendorf-IB-ieq}
 we obtain in our special case the inequality
 $$\left|n_1\frac{\log \alpha}{\log 2}-m_1+ \frac{\log \left(\frac{\alpha^2}{\sqrt 5}\right)}{\log 2}\right|\leq 4\alpha^2 \alpha^{n_2-n_1}.$$
 We apply Lemma \ref{lem:Reduction} to obtain that $n_2-n_1\leq 333$.
 
 Now we consider inequality \eqref{eq:Zeckendorf-IS-ieq} with $K=3$ and $L=2$. Then we obtain in our special case that
 \begin{equation}\label{eq:ZB-red-round2}
  \left|n_1\frac{\log \alpha}{\log 2}-m_1+ \frac{\log \left(\frac{\alpha^2(1+\alpha^{n_2-n_1})}{\sqrt 5}\right)}{\log 2}\right|\leq 4\alpha^2 \alpha^{n_3-n_1}
 \end{equation} 
 and we apply Lemma \ref{lem:Reduction} with $\tau=\frac{\log \alpha}{\log 2}$ and 
 $$\mu=\mu_r=\frac{\log \left(\frac{\alpha^2(1+\alpha^{-r})}{\sqrt 5}\right)}{\log 2}$$
 for all $2\leq r\leq 333$. Note that since we consider Zeckendorf expansions we may exclude the case that $r=n_1-n_2=0,1$. Applying Lemma \ref{lem:Reduction} we obtain that
 for $r\neq 2,6$ we have $n_1-n_3\leq 348$. Let us consider the case $r=2$ and $r=6$ separately. In cases that $r=2$ and $r=6$ we obtain
 $$\mu_2=\frac{\log \alpha}{\log 2} \qquad \text{and}\qquad \mu_6=-\frac{\log \alpha}{\log 2}+1$$
 and inequality \eqref{eq:ZB-red-round2} turns into
 $$\left|(n_1+1)\frac{\log \alpha}{\log 2}-m_1\right|\leq 4\alpha^2 \alpha^{n_3-n_1}$$
 in case that $r=2$ and
 $$\left|(n_1-1)\frac{\log \alpha}{\log 2}-(m_1-1)\right|\leq 4\alpha^2 \alpha^{n_3-n_1}$$
 in case that $r=6$. Let us assume for the moment that $n_1-n_3>348$ then we have that
 $$\left|\frac{m_1}{n_1+1}-\frac{\log \alpha}{\log 2}\right|\leq \frac{4\alpha^2}{n_1+1} \alpha^{n_3-n_1}<\frac{1}{2(n_1+1)^2}$$
 and
 $$\left|\frac{m_1-1}{n_1-11}-\frac{\log \alpha}{\log 2}\right|\leq \frac{4\alpha^2}{n_1-1} \alpha^{n_3-n_1}<\frac{1}{2(n_1-1)^2}$$
 since $\frac{\alpha^{n_1-n_3}}{8\alpha^2}\geq\frac{\alpha^{348}}{8\alpha^2}>3.1 \cdot 10^{64}>n_1\pm 1$. Thus by a criterion of Legendre
 $\frac{m_1}{n_1+1}$ and $\frac{m_1-1}{n_1-1}$ are convergents to $\frac{\log \alpha}{\log 2}$ respectively.
 Therefore we may assume that  $\frac{m_1}{n_1+1}$ and $\frac{m_1-1}{n_1-1}$ are convergents of the form $\frac{p_j}{q_j}$ for some $j=0, 1, 2, \dots, 133$ respectively.
Indeed, we may assume that $j\leq 133$ since $q_{134} > 3.1 \cdot 10^{64} > n_1 + 1$ and we get (see e.g. \cite[page 47]{Baker:NT} or Section \ref{Sec:Implementation}) 
$$\frac{1}{(a_{j+1} + 2)q_j^2} < \left|\frac{\log \alpha}{\log 2} - \frac{p_j}{q_j} \right|,$$
where $[a_0,a_1,\dots]$ is the continued fraction expansion of $\frac{\log \alpha}{\log 2}$. Since 
$$\max \lbrace a_{j+1}: j=0, 1, 2, \dots, 133 \rbrace = 134,$$
we have
$$\frac{1}{ 136 q_j^2} < \frac{4\alpha^2}{q_j} \alpha^{n_3-n_1} $$
and $q_j$ divides one of $\lbrace n_1+1, n_1-1 \rbrace$. Thus we get the inequality
$$ 136 \cdot q_{133} 4\alpha^2 \geq \alpha^{n_1-n_3}$$
which yields $n_1-n_3 < 332$. Thus overall we obtain $n_1-n_3\leq 348$.
 
 Now we consider inequality \eqref{eq:Zeckendorf-IS-ieq} with $K=4$ and $L=2$. Then we obtain in our special case the following inequality
 $$\left|n_1\frac{\log \alpha}{\log 2}-m_1+ \frac{\log \left(\frac{\alpha^2(1+\alpha^{n_2-n_1}+\alpha^{n_3-n_1})}{\sqrt 5}\right)}{\log 2}\right|\leq 4\alpha^2 \alpha^{n_4-n_1}$$
 and we apply Lemma \ref{lem:Reduction} with $\tau=\frac{\log \alpha}{\log 2}$ and 
 $$\mu=\mu_{r,s}=\frac{\log \left(\frac{\alpha^2(1+\alpha^{-r}+\alpha^{-s})}{\sqrt 5}\right)}{\log 2}$$
 for all $2\leq r\leq 333$ and $r+2\leq s \leq 348$. Applying Lemma \ref{lem:Reduction} we obtain that in any case $n_1-n_4\leq 355$ holds.
 
 Finally we consider inequality \eqref{eq:Zeckendorf-IS-ieq} with $K=5$ and $L=2$. Then we obtain 
 $$\left|n_1\frac{\log \alpha}{\log 2}-m_1+ \frac{\log \left(\frac{\alpha^2(1+\alpha^{n_2-n_1}+\alpha^{n_3-n_1}+\alpha^{n_4-n_1})}{\sqrt 5}\right)}{\log 2}\right|\leq 4\alpha^2 \alpha^{-n_1}$$
 and again we apply Lemma \ref{lem:Reduction} with $\tau=\frac{\log \alpha}{\log 2}$ and 
 $$\mu=\mu_{r,s,t}=\frac{\log \left(\frac{\alpha^2(1+\alpha^{-r}+\alpha^{-s}+\alpha^{-t})}{\sqrt 5}\right)}{\log 2}$$
 for all $2\leq r\leq 333$, $r+2\leq s \leq 348$ and $s+2\leq t \leq 355$. Applying Lemma \ref{lem:Reduction} we obtain that in all cases except $(r,s,t)=(3,6,9)$  we have that $n_1\leq 364$.
 The case that $(r,s,t)=(3,6,9)$ can be resolved by using Legendre's
 criterion. Overall we obtain that $n_1\leq 364$.
 
 However due to \eqref{eq:ZB-m1-n1-ieq} we deduce that $m<254$. Hence for every $m=1,\dots,253$ we compute the Zeckendorf expansion of $2^m$ and consider those expansions 
 which have either three or four non-zero digits. In particular we find that
 \begin{align*}
  2^5=32&=F_{6} +F_{4}+F_{2},&  2^8=256&=F_{11} +F_{6}+F_{1},\\
  2^6=64&= F_{8} +F_{4}+F_{0},& 2^{10}=1024&=F_{14}+F_{7}+F_{2},\\
  2^7=128&=F_{9} +F_{7}+F_{3}. 
 \end{align*} 
\end{proof}

\begin{remark}
 The computation time for the reduction process was about $19$ hours and $40$ minutes on a usual PC (Intel Core i7-3770) on a single kernel. In view of solving completely the case that $M=6$
 we would have to solve $5$ Diophantine equations and we would expect that the computation time for the reduction process would be roughly $300$ times longer for each Diophantine equation. 
 Thus we expect a computation time of about $40$ months on a single computer with a single
 kernel. Let us note that the process can be parallelized and using a local network with several computers such a computation seems to be feasible. 
\end{remark}

\section*{Acknowledgment}
The author wants to thank Kwok Chi Chim for many helpful discussions and suggestions.
 

\def\cprime{$'$}

\end{document}